\documentclass[a4paper, 11pt]{amsart}
\usepackage{amsmath,amstext,amsthm,a4,amssymb,amscd}   
\usepackage[francais]{babel}
\usepackage[T1]{fontenc}
\usepackage[utf8]{inputenc}
\usepackage{lmodern}
\usepackage[all]{xy}
\usepackage{hyperref}
\usepackage{verbatim}
\usepackage{cancel}
\usepackage{tikz-cd}

\usetikzlibrary{arrows}
\input xy
\xyoption{all}
\theoremstyle{plain}
\newtheorem{theo}{Théorème}[section]
\newtheorem{prop}[theo]{Proposition}
\newtheorem{lemm}[theo]{Lemme}
\newtheorem{coro}[theo]{Corollaire}

\theoremstyle{definition}
\newtheorem{defi}[theo]{Definition}

\newtheorem{conj}[theo]{Conjecture}
\newtheorem{rema}[theo]{Remarque}

\newcommand{\mf}{\mathfrak}

\newcommand{\A}{\mathfrak{A}}
\newcommand{\B}{\mathfrak{B}}
\newcommand{\p}{\mathfrak{P}}

\title{...}
\date {\today}
\author{}

\def\Q{{\bf Q}}
\def\Z{{\bf Z}}
\def\C{{\bf C}}
\def\N{{\bf N}}
\def\R{{\bf R}}

\def\zp{{\Z_p}}

\def\qp{{\Q_p}}

\begin{document}
\title{Représentations supercuspidales:\\
Comparaison des constructions de Bushnell-Kutzko et Yu   }
 \author{Arnaud Mayeux}
\maketitle

\section*{Résumé}

  Ce texte est une réponse à la question suivante: Quelles sont les méthodes pour construire des représentations complexes supercuspidales des groupes réductifs $p-$adique et existe-il des liens entre elles ? On donnera un aperçu des constructions de Bushnell-Kutzko et de celle de Yu. Puis, étant donnée une strate simple maximale modérée, un caractère simple $\theta$ associée à cette strate et une représentation de niveau zéro d'un sous groupe de $G$, on associe  une donnée de Yu générique et donc une représentation construite par Yu, cette représentation correspond à un type simple donné par une $\beta - extension $ de $\theta$.
 
\section*{Introduction}

Soit $G$ un groupe algébrique réductif connexe défini sur un corps local non archimédien $F$ de caractéristique résiduelle $p$. Notons $\mathcal{M}(G)$ la catégorie des représentations lisses admissibles de $G(F)$. Si $M$ est un sous-groupe de Levi de $G$, il existe un foncteur d'induction $i_M^G:\mathcal{M}(M) \to \mathcal{M}(G)$. Une représentation est dite supercuspidale si elle ne se réalise pas comme un sous quotient d'une représentation $ i_M^G(\pi)$ où $\pi \in \mathcal{M}(M)$ et $M$ est un sous-groupe de Levi propre de $G$. Ainsi les représentations supercuspidales de $G(F)$ sont celles qui ne proviennent pas de représentations d'un sous-groupe de Levi propre de $G(F)$. Les construire est donc un problème fondamental.
Les méthodes les plus abouties pour construire des représentations supercuspidales sont d'une part celles de Adler-Kim-Yu  et d'autre part celles de Bushnell-Kutzko-Sécherre-Stevens. 

La construction de Yu (2001) fournit pour un groupe algébrique réductif connexe arbitraire des représentations supercuspidales, ces représentations ont un trait modéré comme on le verra par la suite. L'immeuble de Bruhat-Tits \cite{brti} et les filtrations de Moy-Prasad sont utilisés de manière cruciale. La construction n'est exhaustive que lorsque $p$ est suffisamment grand en comparaison avec la taille du groupe. La construction de Yu est en partie inspirée des travaux de Adler.

La construction de Bushnell et Kutzko (1993) fournit toutes les représentations supercuspidales de
$\mathrm{GL}_{N}$ sans hypothèse sur $p$. Cette construction est élaborée, cela est dû au trait sauvage des extensions de corps qui interviennent. La construction de Stevens (2008) fournit toutes les
représentations supercuspidales des groupes classiques pourvu que $p$ soit impair. Elle est écrite dans
le même langage que la théorie de Bushnell et Kutzko. La construction de Sécherre (2008) fournie toutes les représentations supercuspidales de $\mathrm{GL}_N(D)$. Blasco et Blondel ont construit des représentations supercuspidales de $\mathrm{G}_2$ en utilisant le langage des strates, ces représentations sont modérées. Van Dinh Ngô a construit des représentations supercuspidales des groupes spinoriels.

L'ancêtre commun de ces constructions est la construction de Howe (1977) des représentations
supercuspidales modérées de $\mathrm{GL}_N$ : dans le cas où $p$ ne divise pas $N$ la construction de Howe est exhaustive.  Les constructions de Yu et Bushnell-Kutzko sont des vastes améliorations dans des directions différentes de la construction originelle de Howe. Koch et Zink avaient des constructions analogues pour $D^{*}$.

Concernant l'organisation du document, on commence par rappeler la définition d'une représentation supercuspidale et on énonce une proposition qui invite à construire des représentations supercuspidales irréductibles par induction compacte (§ \ref{resu}). Ensuite, on résume la construction de Bushnell-Kutzko (§\ref{buku}), étant donné certains objets, on définit plusieurs groupes: $H^1(\beta , \A)$, $J^0(\beta,\A)$ et $E^{\times } J^0(\beta,\A)$. Sur le premier groupe est défini par Bushnell-Kutzko un ensemble de caractères $\theta$ dit simples. Sur le second est défini, étant donné un caractère $\theta$ de $H^1(\beta,\A)$, un ensemble de représentations $\kappa$ contenant $\theta$ appelé $\beta -extension$ de $\theta$. On donne alors la définition d'un type simple. Des représentations supercuspidales sont alors obtenues par induction compacte de représentations de $E^{\times } J^0(\beta,\A)$.
On résume ensuite la construction de Yu (§\ref{yu}), étant donné certains objets, Yu définit plusieurs groupes: $K^d_+ $ et $K^d$. Yu définit ensuite un caractère sur $K_+^d$ puis une représentation $\rho _d$ sur $K^d$ contenant ce dernier caractère. L'induite compacte de $\rho _d$ est alors supercuspidale est irréductible.
On montre ensuite certaines propriétés utiles pour donner le lien recherché (§\ref{stsm},\ref{emrs},\ref{egcg}). 
 Étant donnée une strate simple maximale modérée et un caractère simple $\theta$ associée à cette strate et une représentation de niveau zéro d'un sous groupe de $G$, on factorise (§\ref{cscg}) le caractère simple en produit de caractères qui donnent naissance à des caractères génériques. En se donnant de plus une représentation cuspidale d'un certain $\mathrm{GL}_f(k_E)$, on associe (§\ref{yube}) une donnée de Yu générique et donc une représentation construite par Yu, on verra que cette représentation correspond à un type simple donné par une $\beta -extension$  de $\theta$.

  Pour faire ce lien, on a commencé par comparer d'une part les constructions de Howe \cite{HO}, \cite{MO} et Bushnell-Kutzko et d'autre part celles de Howe et Yu, on a ensuite obtenu par concaténation le lien entre la construction de Bushnell-Kutzko et celle de Yu. 
  
  On indiquera dans la suite par les symboles $\widetilde{BKH}$ , $\widetilde{HY}$ et $\widetilde{BKHY}$ que ces stratégies ont été utilisées.

\section*{Remerciements}
  Je remercie ma directrice de thèse Anne-Marie Aubert pour m'avoir permis de consacrer du temps à l'étude des constructions des représentations supercuspidales, je la remercie aussi pour son aide et ses références. Je remercie Guy Henniart pour son soutien et ses conseils. Je remercie Jeffrey Adler, Colin Bushnell et Vincent Sécherre pour avoir répondu à mes questions. Je remercie Corine Blondel, Peter Latham, Fiona Murnaghan et Shaun Stevens  pour diverses discussions intéressantes.

 \section*{Notations}
 
 On fixe $F$ un corps local non archimédien. On fixe $\bar{F}$ une clôture algébrique de $F$. On fixe une uniformisante $\pi _F$ de $F$ et $\psi$ un caractère additif de conducteur l'idéal maximal de l'anneau des entiers $\mathfrak{o} _F$ de $F$.  \\
 
 Si $E/F$ est une extension finie de $F$, et $\pi _E$ est une uniformisante de $E$, on note $\nu _ E$ la valuation sur $E$ tel que $\nu _ E (\pi _ E)=1$. On note $\mf{p} _E $ l'idéal maximal de l'anneau des entiers $\mf{o} _E$ de $E$. On note $e(E\mid F) $ et $f(E\mid F)$ les indices de ramification et d'inertie. On note $\mathrm{ord}$ l'unique valuation de $\bar{F}$ qui prolonge $\nu _ F$.\\
 
  Si $G$ est un groupe réductif défini sur $F$, on note $Z(G)$ le centre de $G$. On note $\mathcal{B}(G,F)$ l'immeuble élargi de Bruhat-Tits, $G(F)$ agit sur $\mathcal{B}(G,F)$. Si $y\in \mathcal{B}(G)$ et $\mathbf{r}\geq 0 \in \tilde{\R}$ (cf \cite[§1]{Yu} pour la définition de $\Tilde{\R }$, on a $\R \cup \R + \subset \tilde{\R}$), on note $G(F)_{y,\mathbf{r}}$ le groupe défini dans les paragraphes $1$ et $2$ de \cite{Yu}, de même on note $\mathfrak{g}(F)_{y,\mathbf{r}} \subset \mathrm{Lie}(G)$ et $\mathfrak{g}(F)^*_{y,\mathbf{r}} \subset \mathrm{Lie}^*(G)$ les filtrations de l'algèbre de Lie et du dual. On note aussi, lorsque $0\leq\mathbf{s}\leq \mathbf{r} \in \tilde{\R}$, $G(F)_{y,\mathbf{s} :\mathbf{r}}=G(F)_{y,\mathbf{s}}/ G(F)_{y,\mathbf{r} }$ et $\mathfrak{g}(F)_{y,\mathbf{s} :\mathbf{r}}=\mathfrak{g}(F)_{y,\mathbf{s}}/ \mathfrak{g}(F)_{y,\mathbf{r} }$.

   Si $y\in \mathcal{B}(G,F)$, on note $[y]$ la projection de $y$ sur l'immeuble réduit $\mathcal{B}^{red}(G,F)$ et $G(F)_{[y]}$ le sous groupe de $G(F)$ qui fixe $[y]$.
  
  Soit $G$ un groupe et $H$ un sous-groupe de $G$. Si $\rho$ est une représentation complexe de $K$ et $g\in G$, on note ${}^g \rho $ la représentation $x\mapsto \rho (g^{-1} x g)$ de $gKg^{-1}$.
  
  On note $I_g(\rho)$ l'espace vectoriel $Hom_{gKg^{-1} \cap K} ({}^g\rho,\rho).$ 
  
  On note $I_G(\rho)=\{ g \in G \mid I_g(\rho) \not = 0 \} \subset G$, on l'appelle l'entrelacement de $\rho $ dans $G$.

\section{Représentations supercuspidales}\label{resu}
 
 Soit $G$ un groupe réductif défini sur $F$ et soit $P=MN$ un sous groupe de parabolique de $G$  ($M$ est un sous-groupe de Levi de $P$ et $N$ est le radical unipotent de $P$). Notons $\mathcal{M}(G)$ la catégorie des représentations lisses irréductibles de $G$. On note $i_P ^G :\mathcal{M}(M) \to \mathcal{M}(G)$ et $r_P^G: \mathcal{M}(G) \to \mathcal{M}(M)$ les foncteurs d'induction et restriction parabolique normalisés. Rappelons la définition d'une représentation supercuspidale:
 
 \begin{defi} Une représentation $\pi \in \mathcal{M}(G)$ est dite supercuspidale si $r_P^G(\pi)=0$ pour tout sous groupe parabolique propre de $G$. 
 \end{defi}
 
 On a alors la proposition classique:
 
 \begin{prop} (Jacquet)
 
 Une représentation $\pi$ est supercuspidale si et seulement si tous ses coefficients matriciels sont à support compact modulo le centre.
 \end{prop}
 \begin{proof} \cite[VI.2]{Rena}
\end{proof}
Soit $K$ un sous groupe ouvert de $G$, on note $\mathrm{c-ind}_K^G$ le foncteur d'induction compacte. 
On déduit alors de la proposition précédente:
\begin{prop} Soit $K$ un sous groupe ouvert et compact modulo le centre de $G$, soit $\rho$ une représentation lisse irréductible de $K$ et soit $\pi=\mathrm{c-ind}_K^G(\rho)$.

Les assertions suivantes sont équivalentes:
\begin{enumerate}
\item[($i$)] $I_G(\rho)=G \setminus K$
\item[($ii$)] $\pi$ est irréductible et supercuspidale.
\end{enumerate}
\end{prop}
 \begin{proof}\cite[1.5]{Cara} \cite[2.1]{reyu}
 \end{proof}
 Cela permet de construire des représentations supercuspidales irréductibles, les constructions que l'on va ensuite étudier sont basées sur ce fait. Il est conjecturé que toutes les représentations supercuspidales irréductibles des groupes réductifs sont construites ainsi:
 
 \begin{conj} Soit $\pi \in \mathcal{M}(G)$ irréductible et supercuspidale. Il existe un sous groupe ouvert $K$ compact modulo le centre et une représentation lisse irréductible $\rho$ de $K$ tel que $\pi=\mathrm{c-ind}_K^G(\rho)$
 \end{conj}
 
 \begin{theo} La conjecture précédente est un théorème dans les cas suivants:
 \begin{enumerate}
\item[($i$)] (\cite{BK} \cite{bukusl} Bushnell-Kutzko 1993)\begin{enumerate} \item  $G=\mathrm{GL}_N$ \item $G=\mathrm{SL}_N$ \end{enumerate}

\item[($ii$)] (\cite{Sech} \cite{Stev} Sécherre, Stevens 2007) \begin{enumerate}\item $G=\mathrm{GL}_N(D)$ 

\item $G$ est un groupe classique  et la caractéristique du corps résiduel de $F$ n'est pas 2.
\end{enumerate}

\item[($iii$)] (\cite{Yu} \cite{Kim} Yu, Kim 2001-2006)

 $G$ est un groupe algébrique réductif connexe arbitraire et la caractéristique du corps résiduel de $F$ est suffisamment grande par rapport au groupe \footnote{cf \cite{Kim} pour un énoncé précis sur les conditions que doit vérifier la caractéristique résiduelle}.

\end{enumerate}

\end{theo}

Les constructions ($i$) et ($ii$) sont écrites dans le langage des strates qui repose sur la théorie algébrique des nombres.

La construction ($iii$) de Yu est écrite dans le langage des groupes algébriques. L'immeuble de Bruhat-Tits (\cite{brti}) et les filtrations de Moy-Prasad (\cite{Mopr}) prennent des places importantes. C'est Adler (\cite{Adle}) qui a utilisé la première fois ces outils pour construire des représentations supercuspidales de groupes réductifs connexes généraux par induction compacte.

L'objectif de ce texte est de comparer les constructions de Bushnell-Kutzko et Yu. On commence par faire un résumé de ces constructions.

\section{La construction de Bushnell et Kutzko}\label{buku}

  Bushnell et Kutzko  \cite{BK} construisent pour chaque représentation supercuspidale irréductible $\pi$ de $\mathrm{GL}_N(F)$, un sous groupe $K$ ouvert compact modulo le centre et une représentation $\Lambda$ de $K$ tels que $\pi = \mathrm{c-ind}_K^G(\Lambda)$. 
  
Il existe de nombreux textes qui résument cette construction, voir par exemple \cite{Bush}. Pour un texte plus historique voir \cite{Henn}.

\subsection{Strates}
Soit $V$ un $F-$espace vectoriel de dimension finie $N$. Soit $A=\mathrm{End}_F(V)$. Si $\A$ est un $\mathfrak{o} _F -$ordre héréditaire dans $A$, on note $\mathfrak{P}$ son radical, et on note $\nu _{\A}$ la valuation sur $\A$ donnée par $\nu _ {\A} (x) =\mathrm{max}\{k \in \Z \mid x \in \mathfrak{P} ^k \} $. Une strate dans $A$ est un quadruplet $[\mf{A},n,r,\beta]$ où $\A$ est un $\mathfrak{o} _F -$ordre héréditaire, $n>r$ sont des entiers et $\beta$ est un élément de $A$ tel que $\nu _{\A} (\beta) \geq -n$. On note $e(\A \mid \mathfrak{o}_F) $ la période de la suite de $\mathfrak{o}_F$-réseaux associée à $\A$ (\cite[1.1.2]{BK}) \\

Une strate est dite pure si:
\begin{enumerate}
\item[($i$)] L'algèbre $E:=F[\beta]$ est un corps.
\item[($ii$)] $E^{\times} \subset \mf{K}(\A)$.
\item[($iii$)] $\nu _{\A}(\beta)=-n$
\end{enumerate}

 Soit $[\mf{A},n,r,\beta]$ une strate pure, on définit pour $k\in \Z$:
 \begin{center}
 $\mf{N}_k(\beta, \A)=\{x \in \A \mid \beta x - x\beta \in \p ^k \}$ 
 \end{center}

 Posons $B=\mathrm{End}_{F[\beta]}$ et $\mathfrak{B}=B \cap \A$ et  définissons un entier $k_0(\beta , \A)$ \footnote{On notera parfois $k_{0F}(\beta,\A)$ si la clarté l'exige} :
 
 -Si $E=F$ on pose $k_0(\beta, \A ) = - \infty$
 
 -Si $E \not = F$ on pose $k_0(\beta, \A )= \mathrm{max}\{ k\in \Z \mid \mf{N}_k \not \subset \mf{B} + \p \}$\\
 
 Une strate $[\mf{A},n,r,\beta]$ est dite simple si elle est pure et si de plus $r<-k_0(\beta,\A)$.
 
 Les strates simples sont construites par itération à partir d'éléments minimaux définis ci-dessous via un procédé qui est l'objet de \cite[2.2,2.2.8,2.4,2.4.1]{BK}. Voici la définition d'un élément minimal donnant lieu à une strate simple ayant "1 itération".
  
 \begin{defi}
 Soit $E/F$ une extension finie.

Un élément $\beta \in E$ est dit minimal sur $F$ si les trois conditions suivantes sont vérifiées:
\begin{enumerate}
\item[$\bullet $] $F[\beta]=E$

\item[$\bullet$] $\mathrm{pgcd}(\nu_E(\beta),e(E\mid ,))=1$

\item[$\bullet$] $\pi_F^{-\nu_E(\beta)}\beta ^{e(E,F)}+\mathfrak{p}_E$ engendre le corps résiduel $k_E$ sur $k_F$.
\end{enumerate}
\end{defi}

 \begin{prop}\label{minider} Soit $[\A,n,n-1,\beta]$ une strate pure dans l'algèbre $A=\mathrm{End}_F(V)$.
 Les assertions suivantes sont équivalentes:
 \begin{enumerate}
 \item [$\bullet$] l'élément $\beta$ est minimal sur $F$
 \item[$\bullet$] $k_0(\beta,\A)=-n$
 \item[$\bullet$] la strate $[\A,n,n-1,\beta]$ est simple
 \end{enumerate}
 \end{prop}
 
 \begin{proof}
 C'est une conséquence directe de \cite[1.4.15]{BK}
 \end{proof}
  \subsection{Corestriction modérée }
  
   Soit $E/F$ une extension finie de $F$ contenue dans $A$. Notons $B=\mathrm{End}_E(V)$ le centralisateur de $E$ dans $A$.
 
 \begin{defi}Une corestriction modérée sur $A$ relativement à $E/F$ est un morphisme de  $(B,B)$-bimodule\footnote{ Un morphisme de $(B,B)-$bimodule $A\to B$ est un morphisme de groupe tel que $s(bab')=bs(a)b'$}  $s:A\to B$ tel que $s(\A)=\A \cap B$ pour tout ordre héréditaire $\A$ dans $A$ normalisé par $E^{\times}$.
 
 \end{defi}
 
 \begin{prop} \cite[1.3.4,1.3.8]{BK} \begin{enumerate} \item[($i$)] Soit $\psi_E$, $\psi_F$ des caractères continus additifs de $E,F$ ayant pour conducteur $\mf{p}_E , \mf{p} _ F$ respectivement. Soient $\psi_B $ et $\psi _A$ les caractères additifs de $B,A$ définis par $\psi_B = \psi _E \circ \mathrm{tr}_{ B/E} $ $\psi_A = \psi_F \circ \mathrm{tr}  _{A/F}$.
 Il existe une unique application $s:A \to B $ tel que $\psi_A ( ab) = \psi _B(s(a)b)$ , $a\in A , b \in B$. L'application $s$ est une corestriction modérée sur $A$ relativement à $E/F$.
 
 \item[($ii$)] Supposons que $E/F$ est une extension modérément ramifiée, alors il existe une corestriction modérée tel que $s\mid _B = \mathrm{Id} _B$
 \end{enumerate}
 \end{prop}

 \subsection{Caractères simples}
 
  A toute strate simple $[\mf{A},n,r,\beta]$ est associé un groupe $H^1(\beta,\A)$ et un ensemble de caractères (dit simples) de $H^1(\beta,\A)$ ayant un entrelacement  dans $G$ remarquable, c'est l'objet de ce paragraphe.
  
  Deux strates de la forme $[\A,n,r,\beta _1 ] $ et $[\A,n,r,\beta _ 2]$ sont dites équivalentes si $\beta_1 -\beta _2 \in \p ^{-r}$, on note alors $[\A,n,r,\beta _1 ] \sim [\A,n,r,\beta _ 2]$.
  
Le théorème suivant est fondamental:

\begin{theo}\label{approxi} \begin{enumerate}
\item[($i$)] Soit $[\A,n,r,\beta]$ une strate pure dans l'algèbre $A$. Il existe une strate simple $[\A,n,r,\gamma]$ dans $A$ tel que 

\begin{center}

$[\A,n,r,\gamma]\sim [\A,n,r,\beta] .$

\end{center}

De plus, si  $[\A,n,r,\gamma]$ est une strate  satisfaisant cette condition alors $e(F[\gamma]\mid F)$ divise $e(F[\beta]\mid F) $ et $f(F[\gamma]\mid F)$ divise $f(F[\beta] \mid F)$.

\item[($ii$)] Soit $[\A,n,r,\beta]$ une strate pure dans $A$ tel que $r=-k_0(\beta,\A)$. Soit $[\A,n,r,\gamma]$ une strate simple dans $A$ qui est équivalente à $[\A,n,r,\beta]$, soit $s_{\gamma}$ une corestriction modérée sur $A$ relativement à $F[\gamma]/F$, soit $B_{\gamma}=\mathrm{End}_{F[\gamma]}(V),$ et $\mf{B} _{\gamma} = \A \cap B_{\gamma}.$ Alors $[\mf{B}_{\gamma},r,r-1,s_{\gamma}(\beta - \gamma)]$ est équivalent à une strate simple dans $B_{\gamma}$.

\end{enumerate}
\end{theo} 
\begin{proof}
\cite[2.4.1]{BK}
\end{proof}

Ainsi le théorème précédent permet, étant donnée une strate pure $[\A,n,r,\beta]$, de lui associer un entier $s$ et une suite de strate simple $[\A,n,r,\beta _i ]$, $0\leq i \leq s$ appelée suite d'approximation et vérifiant:
\begin{enumerate}
\item[($i$)]$[\A , n , r_i , \beta _i ]$  est simple pour $0\leq i \leq s$;
\item[($ii$)]$[\A,n,r_0,\beta _0 ] \sim [\A , n , r , \beta ];$
\item[($iii$)]$r=r_0 <r_1<\ldots <r_s <n ;$
\item[($iv$)]$r_{i+1}=-k_0(\beta _ i , \A), et [\A, n , r_{i+1} , \beta _{i+1} ]$ est équivalent à $[\A,n,r_{i+1} , \beta _i] $ pour $0 \leq i \leq s-1$;
\item[($v$)]$k_0(\beta _ s , \A) =-n $ ou $- \infty ;$
\item[($vi$)] Soit $\A _ i $ le centralisateur de $\beta _ i$ dans $\A$ et $s_i$ une corestriction modérée sur $A$  relativement à $F[\beta _ i ] / F$. La strate dérivée $[\A _ {i+1} , r_{i+1} , r_{i+1} -1 , s_{i+1}(\beta _{ i} - \beta _{i+1} )]$ est équivalente à une strate simple pour $ 0 \leq i \leq s-1$.
\end{enumerate}
\begin{rema} Le procédé d'approximation termine après un nombre fini d'étape puisque le degré de la suite des corps $F[\beta _ i ]$ est strictement décroissante d'après \ref{approxi}, et qu'un élément $\beta$ appartenant à $F$ vérifie $k_0 (\beta,\A) = - \infty$. 
\end{rema}

On peut maintenant définir les groupes sur lesquels les caractères simples et les types simples seront définis, fixons une strate simple $[\A,n,0,\beta],$ et posons $r=-k_0(\beta,\A)$.

\begin{defi}\label{defsimpl} \begin{enumerate}
\item[($i$)] Supposons que $\beta$ est minimal sur $F$.
On pose \begin{enumerate} \item $\mf{H}(\beta,\A)=\mf{H}(\beta,\A)=\mf{B}_{\beta} + \p ^{[\frac{n}{2}]+1}$
\item $\mf{J}(\beta,\A)=\mf{J}(\beta,\A)=\mf{B}_{\beta} + \p ^{[\frac{n+1}{2}]}$ \end{enumerate}

\item[($ii$)] Supposons que $r<n$, soit $[\A,n,r,\gamma]$ une strate simple équivalente à $[\A,n,r,\beta].$ 

On pose \begin{enumerate} \item $\mf{H}(\beta,\A)=\mf{B}_{\beta} + \mf{H}(\gamma) \cap \p ^{[\frac{r}{2}]+1}$
 \item $\mf{J}(\beta,\A)=\mf{B}_{\beta} + \mf{J}(\gamma) \cap \p ^{[\frac{r+1}{2}]}$
\end{enumerate}

\item[($iii$)] Posons alors \begin{enumerate}
\item $\mf{H}^k(\beta,\A)=\mf{H}(\beta,\A)\cap \p ^k$

\item $\mf{J}^k(\beta,\A)=\mf{J}(\beta,\A)\cap \p ^k$
\end{enumerate}
\item[($iv$)] Enfin, notons $U^m(\A)=\left(1+\mathfrak{P}^m \right)$ et posons \begin{enumerate}
\item $H^m(\beta, \A)=\mf{H}(\beta,\A) \cap U^m(\A) $
\item $J^m(\beta, \A)=\mf{J}(\beta,\A) \cap U^m(\A)$
\end{enumerate}
\end{enumerate}
 \end{defi} 
 
 \begin{rema}D'après \cite[3.1.7]{BK}, $\mf{J}^k(\beta,\A)$ et  $\mf{H}^k(\beta,\A)$ sont bien définis, ils ne dépendent pas du choix de $\gamma$. Il en est donc de même de $H^m(\beta, \A)$ et $J^m(\beta, \A)$.

 \end{rema}
 
  On note $\psi _A$ la fonction sur $A$ définie par $\psi_{A}(x)=\psi \circ \mathrm{tr}_{A/F} (x)$. Soit $b\in A$ on définit la fonction  $\psi _ b$ sur $A$ par :
 \begin{center}
 
 $\psi _{b}(x)=\psi _A (b (x-1)$
 
 \end{center}
 
 Nous pouvons désormais donner la définition des caractères simples de Bushnell-Kutzko.
 
 \begin{defi}\begin{enumerate}\item Supposons que $\beta $ est minimal sur $F$. Soit $0 \leq m \leq n-1$, soit $\mathcal{C}(\A,m,\beta)$ l'ensemble des caractères $\theta$ de $H^{m+1}(\beta)$ tel que:
 
 \begin{enumerate}
 \item $\theta \mid H^{m+1}(\beta) \cap U^{[\frac{n}{2}]+1}(\A)=\psi_{\beta}$
 \item $\theta \mid H^{m+1}(\beta) \cap B_{\beta}^{\times}$ factorise par $\mathrm{det}_{B_{\beta}}:B_{\beta}^{\times} \to F[\beta]^{\times}.$
 \end{enumerate}
 
 \item Supposons que $r<n. $ Soit $0\leq m \leq r-1$, soit $\mathcal{C}(\A,m,\beta)$ l'ensemble des caractères $\theta$ de $H^{m+1}(\beta)$ tel que \begin{enumerate}
 
 \item $\theta \mid H^{m+1}(\beta) \cap B_{\beta}^{\times}$ factorise par $\mathrm{det}_{B_{\beta}}$
 \item $\theta$ est normalisé par $\mf{K}(\mf{B}_{\beta})$
 \item si $m'=\mathrm{max}\{m,[\frac{r}{2}\},$ la restriction $\theta \mid H^{m'+1}(\beta)$ est de la forme $\theta _0 \psi _c $ avec $\theta _0 \in \mathcal{C}(\A,m',\gamma)$ et $c=\beta -\gamma.$

 \end{enumerate}
 \end{enumerate}
  \end{defi}
 \begin{rema} Dans le deuxième cas, $\mathcal{C}(\A,m,\beta)$ est défini par induction: étant donné une suite d'approximation $[\A , n , r_i , \beta _i ]$ ,$0\leq i \leq s$ de $[\A,n,0,\beta]$, le dernier terme de la suite est tel que $\beta _ s $ est minimal sur $F$, on définit ainsi un ensemble de caractère simple associé, puis on définit ceux associés à $[\A , n , r_{s-1} , \beta _{s-1} ]$ et par itération on obtient les caractères simples $\mathcal{C}(\A,m,\beta)$.
\end{rema}
 \begin{prop} Soit $[\A,n,0,\beta]$ une strate simple dans l'algèbre $A$, posons $r=-k_0(\beta,\A)$. Soit $0\leq m \leq [\frac{r}{2}]$, et $\theta \in \mathcal{C}(\A,m,\beta).$ Alors 
 \begin{center}
 $I_{G}(\theta)=J^{[\frac{r+1}{2}]}(\beta,\A) B_{\beta}^{\times} J^{[\frac{r+1}{2}]}(\beta,\A)$
 \end{center}
 \end{prop}
 \begin{proof}
 \cite[3.3.2 Remark]{BK}
 \end{proof}
 \subsection{Types simples et représentations}
 On donne dans ce paragraphe la définition d'un type simple au sens de Bushnell-Kutzko. Enfin on énonce le théorème principal de Bushnell-Kutzko. \\
 
 Soit $[\mf{A},n,0,\beta]$ une strate simple et soit $\theta$ un caractère simple de $H^1(\beta, \A)$. Il existe une unique représentation irréductible $\eta$ de $J^1(\beta, \A)$ contenant $\theta$ \cite[5.1.1]{BK}.\\
 
 \cite[5.2.1]{BK}Une $\beta -extension$ de $\theta$ est une représentation $\kappa$ de $J^0(\beta , \A)$ tel que :
 
 \begin{enumerate}
 \item[($i$)] $\kappa \mid J^1(\beta)= \eta $
 \item[($ii$)] $\kappa$ est entrelacé par $B^{\times} $
 
 \end{enumerate}
 
  \begin{prop} \label{beta} Soit $\kappa$ une représentation de $J^0(\beta , \A)$. Les assertions suivantes sont équivalentes:\begin{enumerate}
  \item[($i$)] La représentation $\kappa$ est une $\beta -extension$ de $\theta$.
  \item[($ii$)] La représentation $\kappa$ satisfait les 3 conditions suivantes \begin{enumerate}
  \item $\kappa $ contient $\theta$
  \item $\kappa$ est entrelacé par $B^{\times} $
  \item $dim(\kappa)=dim (\eta ) = [J^1(\beta , \A) : H^1(\beta ,\A )]^{\frac{1}{2}}$
  \end{enumerate}
  \end{enumerate}
  \end{prop}
  
  \begin{proof} Si $\kappa$ est une $\beta-extension$, $\kappa$ satisfait $(a),(b),(c)$. Réciproquement si $\kappa$ satisfait $(a),(b),(c)$ alors ${(\kappa \mid _{J^1 (\beta , \A)}}) \mid _{H^1(\beta , \A )}$ contient $\theta $ donc  $\kappa \mid _{J^1 (\beta , \A )}$ contient $\eta$ et vu l'égalité des dimensions $\kappa \mid _{J^1 (\beta , \A) }= \eta$, ainsi $\kappa$ est une $\beta -extension$.
  \end{proof}

\begin{defi} Un type simple est ou bien 
\begin{enumerate}
\item une représentation irréductible $\lambda = \kappa \otimes \sigma $ de $J(\beta, \A )$ où:
\begin{enumerate}
\item $\A$ est un $\mathfrak{o} _F -$ordre héréditaire principal dans $\A$ et $[\A,n,0,\beta]$ est une strate simple;
\item $\kappa$ est une $\beta -extension$ d'un caractère $\theta \in \mathcal{C}(\A,0,\beta)$
\item Si l'on pose $E=F[\beta],\mf{B} =\A \cap \mathrm{End}_{E}(V)$, de telle sorte que :
\begin{center}
$J(\beta, \A ) / J^1(\beta ,\A) \simeq U(\mf{B})/U^1(\mf{B})\simeq \mathrm{GL}_f(k_E)^e$

\end{center}

pour certains entiers $e$ et $f$,

alors $\sigma$ est l'inflation d'une représentation $\sigma _0 \otimes \cdots \otimes \sigma _0 $ où $\sigma _0 $ est une représentation irréductible cuspidale de $\mathrm{GL}_f(k_E)$
\end{enumerate} ou bien
\item une représentation irréductible $\sigma$ de $U(\A)$ où:
\begin{enumerate}
\item $\A$ est un $\mathfrak{o} _ F -$ordre principal dans $A$
\item Si l'on écrit $U(\A)/U^1(\A) \simeq \mathrm{GL}_f(k_F)^e, $ pour certains entiers $e,f$, alors $\sigma$ est l'inflation d'une représentation $\sigma_0 \otimes \cdots \otimes \sigma _0$, où $\sigma _0$ est une représentation irréductible cuspidale de $\mathrm{GL}_f(k_F)$
\end{enumerate}
 \end{enumerate}
\end{defi}
 
 On a alors le théorème suivant:
 
 \begin{theo}
 Soit $\pi$ une représentation irréductible supercuspidale de $G=\mathrm{Aut}_F(V)\simeq \mathrm{GL}_N(F)$. Il existe un type 
 simple $(J,\lambda)$ dans $G$ tel que $\pi \mid J $ contient $\lambda$. De plus,
 \begin{enumerate}
 \item[($i$)] le type simple $(J,\lambda)$ est uniquement déterminé à G-conjugaison près.
 \item[($ii$)] si $(J,\lambda)$ est donné par une strate simple $[\A,n,0,\beta]$ dans $A=\mathrm{End}_F(V)$ avec $E=F[\beta]$, il existe une unique représentation $\Lambda$ de $E^{\times} J$ telle que $\Lambda \mid J = \lambda $ et $\pi=\mathrm{c-ind}(\Lambda)$
 \item[($iii$)] Si $(J,\lambda)$ est de type (2), c'est à dire, $J=U(\A)$ pour un ordre héréditaire principal $\A$ et $\lambda$ est trivial sur $U^1(\A)$, alors il existe une unique représentation $\Lambda$ de $F^{\times} U(\A)$ telle que $\Lambda \mid U(\A) = \lambda $ et $\pi = \mathrm{c-ind} (\Lambda)$.
 \end{enumerate}
 \end{theo}

 \section{La construction de Yu}\label{yu}

 Étant donné le groupe des $F-$points d'un groupe algébrique réductif connexe arbitraire, Yu \cite{Yu} construit des représentations supercuspidales irréductibles "modérées" de $G$. Kim \cite{Kim} a montré que lorsque la caractéristique résiduelle de $F$ est suffisamment grande la construction de Yu est exhaustive. 
 
  Yu se donne plusieurs objets:
 
 \begin{enumerate}
 \item[($\overrightarrow{G}$)]  Une suite strictement croissante de groupes algébriques $G^0\subset \cdots\subset G^i \subset \cdots \subset G^d $ tel que 
 \begin{enumerate}
 \item Il existe une extension finie modérément ramifiée galoisienne $E/F$ tel que $G^i \otimes E$ est un sous groupe de Levi déployé de $G \otimes E$, on appellera une telle suite " suite de Levi tordue modérée ".
 \item $Z(G^0)/Z(G)$ est anisotrope.
 \end{enumerate}
 \item[($y$)] un point $y \in \mathcal{B}(G^0,F) \cap A(G,T,E)$ où $T$ est un tore maximal de $G^0$, tel que $T \otimes E $ est déployé et $A(G,T,E)$ désigne l'appartement associé à $T$ sur $E$. 
 \item[($\overrightarrow{r}$)] une suite de nombres réels $0<\mathbf{r}_0<\mathbf{r}_1<...<\mathbf{r}_{d-1}\leq \mathbf{r}_d$ si $d>0$ , $0\leq \mathbf{r}_0$ si $d=0$
 \item[($\rho$)] une représentation irréductible $\rho$ de $K^0=G^0_{[y]}$ telle que $\rho \mid G^0_{y,0+}(F)=1$ et telle que $\pi_0:=\mathrm{c-ind}_{K^0}^{G^0(F)}(\rho)$ est irréductible supercuspidale
 \item[($\overrightarrow{\boldsymbol{\phi}}$)] une suite $\boldsymbol{\phi} _0,\ldots, \boldsymbol{\phi} _d $ de caractères de $G^0(F),\ldots,G^d(F)$. On suppose que $\boldsymbol{\phi} _i $ est trivial sur $G^i(F)_{y,\mathbf{r}_i+} $ mais pas sur $G^i(F)_{y,\mathbf{r}_i} $ pour $0\leq i \leq d-1$. Si $\mathbf{r}_{d-1} < \mathbf{r}_d$, on suppose que $\boldsymbol{\phi} _d$ est trivial sur $G^d(F)_{y,\mathbf{r}_d +}$ mais pas sur $G^d(F)_{y,\mathbf{r}_d}$.
 
 \end{enumerate}
 
 On appellera un tel quintuplet une donnée de Yu, une donnée de Yu sera dite générique si $\boldsymbol{\phi} _i $ est $G_{i+1}-$générique de profondeur $\mathbf{r}_i$, cette dernière notion sera définie dans la suite de ce texte au paragraphe \ref{egcg}.\\
 
 Fixons dans ce paragraphe une donnée de Yu générique.

Les trois premières données permettent de définir directement trois groupes, sur le premier sera ensuite défini un caractère, sur les suivants sera définie une représentation.

 \begin{defi} Soit $\mathbf{s}_i=\frac{\mathbf{r}_i}{2}$.
 
  Posons \cite[§3, 15.3]{Yu} \begin{enumerate} 
 
 \item[($i$)] $K^d _+ = G^0(F)_{y,0+} G^1(F)_{y,\mathbf{s}_0+} \cdots G^d(F)_{y,\mathbf{s}_{d-1}+} $
 
 \item[($ii$)] $^{\circ} K^d=G^0(F)_{y,0} G^1(F)_{y,\mathbf{s}_0} \cdots G^d(F)_{y,\mathbf{s}_{d-1}}$

 \item[($iii$)] $ K^d=G^0(F)_{[y]} G^1(F)_{y,\mathbf{s}_0} \cdots G^d(F)_{y,\mathbf{s}_{d-1}}$

 \end{enumerate}
 \end{defi}
 \begin{rema}
 On a des inclusions : $K_+^d ~\subset ~ ^{\circ}K^d ~\subset ~K^d$. Les groupes $K_+^d$ et $ ^{\circ} K^d$ sont compacts et $  K^d$ et compact modulo le centre. Le groupe  $ K^d$ a un unique sous groupe compact maximal, il s'agit de $^{\circ} K^d$.
 \end{rema}
 
 Grâce à $\overrightarrow{\boldsymbol{\phi}}$, Yu définit un caractère $\prod\limits_{i=1}^d \hat{\boldsymbol{\phi}_i}$ sur $K^d _+$ \cite{Yu}. Il définit ensuite une représentation $\rho _d=\rho _d( \overrightarrow{G},y,\overrightarrow{\mathbf{r}},\rho ,\overrightarrow{\boldsymbol{\phi}}) $ de $K^d$ \cite[§4]{Yu}.\\
 
 Expliquons maintenant comment sont construits ces objets.\\
 
 Posons $T^i=(Z(G^i))^{\circ}$, et considérons l'action adjointe de $T^i$ sur $\mathfrak{g}$, alors $\mathfrak{g}^i=\mathrm{Lie}(G^i)$ est le sous espace maximal sur lequel $T^i$ agit trivialement. Soit $\mathfrak{n}^i$ la somme des autres sous espaces propres.
 Soit $\mathbf{s} \geq 0 \in \tilde{\R}$, alors $\mathfrak{g}(F) _s = \mathfrak{g} ^i (F) _s \oplus \mathfrak{n} ^i (F) _s $ où $\mathfrak{n} ^i (F) _s \subset \mathfrak{n} ^i (F)$.
  
 \begin{rema}\label{nigl}Si $G=\mathrm{GL}_{N}$, alors $\mathfrak{g}=\mathrm{M}_N(F)$ et $\mathfrak{g}$ est muni de la forme bilinéaire symétrique $(x,y)\mapsto \mathrm{tr}(xy)$ et on a $\mathfrak{g}(F)_{\mathbf{s}}=\mathfrak{g}^i(F)_{\mathbf{s}} \oplus \left((\mathfrak{g}^i(F))^{\perp} \cap \mathfrak{g}(F) _{\mathbf{s}} \right)$ \cite[Lemma 3]{HO}. On a dans ce cas $\mathfrak{n}^i(F) _{s} = (\mathfrak{g}^i(F))^{\perp} \cap \mathfrak{g}(F) _{\mathbf{s}}$.
 
 \end{rema}

 On a des isomorphismes:
 
 \begin{center}
 $
 G^i(F)_{\mathbf{s}_{i+} :\mathbf{r}_{i+}} \simeq \mathfrak{g}^i(F)_{\mathbf{s}_{i+} :\mathbf{r}_{i+}} \subset \mathfrak{g}^i(F)_{\mathbf{s}_{i+} :\mathbf{r}_{i+}} \oplus \mathfrak{n}^i(F)_{\mathbf{s}_{i+} :\mathbf{r}_{i+}}\simeq G(F)_{\mathbf{s}_{i+} :\mathbf{r}_{i+}}
 $
 \end{center}
 
 Le caractère $\boldsymbol{\phi } _i$ de $G^i(F)$ est de profondeur $\mathbf{r} _i $ il induit donc, via le premier isomorphisme ci dessus,  un caractère sur $\mathfrak{g}^i(F)_{\mathbf{s}_{i+} :\mathbf{r}_{i+}}$, on prolonge ce dernier à  $\mathfrak{g}^i(F)_{\mathbf{s}_{i+} :\mathbf{r}_{i+}} \oplus \mathfrak{n}^i(F)_{\mathbf{s}_{i+} :\mathbf{r}_{i+}}$ en décrétant qu'il vaut $1$ sur $\mathfrak{n}^i(F)_{\mathbf{s}_{i+} :\mathbf{r}_{i+}}$, on obtient alors via le dernier isomorphisme un caractère de $G(F)_{\mathbf{s}_{i+}}$ que Yu note $\hat{\boldsymbol{\phi}_i}$. Par construction, $\hat{\boldsymbol{\phi}_i} \mid _{G^i(F)_{\mathbf{s}_{i+}}}= \boldsymbol{\phi} _i \mid _{G^i(F)_{\mathbf{s}_{i+}}}$. Il existe alors un unique caractère de $G^0(F)_{[y]}G^i(F)_0 G(F)_{\mathbf{s} _{i+}}$ qui étend $\boldsymbol{\phi} _i$ et $\hat{\boldsymbol{\phi} _i}$. Yu note encore ce prolongement $ \hat{\boldsymbol{\phi} _i}$. Remarquons que 
 $K_+^d \subset G^0(F)_{[y]}G^i(F)_0 G(F)_{\mathbf{s} _{i+}}$, on a donc défini en particulier un caractère $ \hat{\boldsymbol{\phi} _i}$ sur $K_+^d$.\\

 Enfin, on construit pour $0\leq i \leq d-1$ une représentation $\kappa _i $ sur $K^d$ en utilisant la représentation de Heisenberg, la représentation $\kappa _i$ ne dépend que de $\boldsymbol{\phi} _i $, la\footnote{Il se peut que $[J^{i+1}:J_+^{i+1}]=1$.} dimension de $\kappa _i$ vaut $[J^{i+1}:J_+^{i+1}]^{\frac{1}{2}}$ où $J^{i+1}$ et $J^{i+1}_+$ sont définis au paragraphe $4$ de \cite{Yu}.  On réfère à \cite[3.25]{hamu} pour la construction de $\kappa _i$. On pose aussi $\kappa _d = \boldsymbol{\phi} _d \mid _{K^d}$.\\
 
 La représentation $\rho$ de $G^0(F)_{[y]}$ s'étend à $K^d$ via l'isomorphisme $K^d /K^d_+ \simeq G^0(F)_{[y]} / G^0(F)_{y,0}$, puisque $\rho \mid _{G^0(F)_{y,0}} =1$ par hypothèse.\\

 La représentation $\rho ^d $ construite par Yu est alors $\rho \otimes \kappa _0 \otimes \kappa _1 \otimes \cdots \otimes \kappa _d$.\\
 
 On pose aussi $\lambda= \kappa _0 \otimes \kappa _1 \otimes \cdots \otimes \kappa _d$, de tel sorte que $\lambda = \rho \otimes \lambda $. La représentation $\lambda$ ne dépend pas de $\rho$ ainsi, étant donné une donnée de Yu générique, on note:

$\bullet$ $\lambda=\lambda( \overrightarrow{G},y,\overrightarrow{\mathbf{r}}, \overrightarrow{\boldsymbol{\phi}})$.

$\bullet $ $\rho _d( \overrightarrow{G},y,\overrightarrow{\mathbf{r}},\rho ,\overrightarrow{\boldsymbol{\phi}})$.
 
 \begin{prop} \cite[4.4]{Yu} La représentation $\rho _d \mid _{K^d _+} contient \prod\limits_{i=0}^d\hat{ \boldsymbol{\phi} _i }\mid _{K_+ ^d}$.
 \end{prop}

 \begin{theo}\begin{enumerate} \item[(Yu)]\cite[4.6 , §15]{Yu} La représentation $\mathrm{c-ind}_{K^d}^{G(F)}\rho_d$ est irréductible et supercuspidale.
 \item[(Kim)]\cite[19.1]{Kim} Si la caractéristique du corps résiduel de $F$ est suffisamment grande alors pour toute représentation supercuspidale $\pi$ de $G(F)$, il existe $(\overrightarrow{G},y,\overrightarrow{\mathbf{r}},\rho ,\overrightarrow{\boldsymbol{\phi}})$, tel que $\pi=\mathrm{c-ind}_{K^d}^{G(F)} \rho _d( \overrightarrow{G},y,\overrightarrow{\mathbf{r}},\rho ,\overrightarrow{\boldsymbol{\phi}})$.
 \end{enumerate} \end{theo}

Posons ${}^{\circ}\rho_d={}^{\circ}\rho_d( \overrightarrow{G},y,\overrightarrow{\mathbf{r}},\rho , \overrightarrow{\boldsymbol{\phi}})=\rho_d \mid _{^{\circ}K^d}$ et  ${}^{\circ}\lambda ={}^{\circ} \lambda ( \overrightarrow{G},y,\overrightarrow{\mathbf{r}}, \overrightarrow{\boldsymbol{\phi}}) = \lambda \mid _{^{\circ}K^d}$.

Posons aussi ${}^{\circ} \kappa _i = \kappa _i  \mid _{ {}^{\circ} K_d }$.
 
 \section{Strate simple modérée}\label{stsm}

On appellera strate simple modérée une strate simple $[\A,n,r,\beta]$ tel que $F[\beta]$ est une extension modérément ramifiée de $F$. On prouve dans cette section quelques propositions  qui seront utiles pour donner le lien recherché.\\

 Soit $[\mathfrak{A},n,r,\beta]$ une strate pure dans l'algèbre $A=\mathrm{End}_F(V)$ tel que $F[\beta]$ est modérément ramifiée sur $F$, notons $E$ le corps $F[\beta]$. Soit $s:A\to B_E$ la corestriction modérée qui est l'identité sur $B_{\beta}=B_E=\mathrm{End}_{F[\beta]}(V)$, on notera donc $s(b)$ par $"b"$ lorsque $b$ appartient à $B_E$.   Notons $\mathfrak{P}$ le radical de Jacobson de $\mathfrak{A}$, notons $\mathfrak{B}_E$ l'intersection de $\mathfrak{A}$ et $B_E$, enfin notons $\mathfrak{Q}_E$ l'intersection de $\mathfrak{P}$ et $B_E$. Ainsi $\B _E$ est un $\mf{o}_{E}-$ordre héréditaire et  $\mf{Q}_E$ est le radical de Jacobson de $\mf{B} _E$ .  \\
  
  Voici une proposition analogue à celle de \cite[2.2.3]{BK}.
\begin{prop} \label{modif2.2.3}Supposons que $[\mathfrak{A},n,r,\beta]$ est simple. Soit $b\in \mathfrak{Q}_E^{-r}$, et supposons que la strate $[\mathfrak{B}_E,r,r-1,b]$ est simple.  Alors:
\begin{enumerate}
\item[($i$)]La strate $[\mathfrak{A},n,r-1,\beta+b]$ est simple.

\item[($ii$)]Le corps $F[\beta +b]$ est égal au corps $F[\beta , b ]$.

\item[($iii$)]$k_0(\beta +b,\mathfrak{A})= \left\{ \begin{array}{ll}
        -r=k_0(b,\mathfrak{B}_E)$ si $b \not \in E \\
        k_0(\beta , \mathfrak{A})$ si $b\in E
    \end{array}  
\right.$

\end{enumerate}
\end{prop}

\begin{proof} Posons $E_1=F[\beta , b ]$. Soit $\iota _W$ une $(W,E_1)$-décomposition de $A$ \cite[1.2.6]{BK}.

Notons $\mathfrak{A}(E_1)=\{ x\in \mathrm{End}_F(E_1) \mid x(\mathfrak{p}_{E_1} ^i) \subset \mathfrak{p} _{E_1} ^i ~~\forall i \in \Z \}$, notons $B_{E_1}=\mathrm{End}_{E_1}(V)$ et  $\mathfrak{B}_{E_1}=\mathfrak{A}\cap \mathrm{End}_{E_1}(V)$. 

Soit $\mathcal{L}=\{L_i\}_{i\in \Z} $ la suite de $\mathfrak{o}_{F}$-réseaux telle que $\A=\{x\in A \mid x(L_i)\subset L_i ~\forall i \in \Z \}$, alors $\B _E=\{x\in B_E \mid x(L_i)\subset L_i ~\forall i \in \Z \}$. Remarquons que par définition \cite[1.2.1 ,1.2.4]{BK} $\mf{K}(\A)=\{x\in G \mid x(L_i) \in \mathcal{L}~ \forall i \in \Z\},$ et $\mf{K}(\B _E)=\{x\in G_E \mid x(L_i) \in \mathcal{L}~ \forall i \in \Z\}.$ Donc $\mf{K}(\B _{E})\subset \mf{K}(\A)$, donc par hypothèse $\beta , b \in \mf{K}(\A),$ donc $E_1^{\times}=F[\beta , b ]^{\times} \subset \mf{K}(\A).$

D'après \cite[1.2.8]{BK} la $(W,E_1)-$décomposition se  restreint en une $(W,E_1)$-décomposition de $\A$.

Posons $B_E(E_1)=End_E(E_1)$ et $\mathfrak{B}_E(E_1)= B_E(E_1) \cap \A _{E_1}$.

Posons $n(E_1)=n/e(\mathfrak{B}_{E_1}\mid \mathfrak{o}_{E_1})$, de même posons $r(E_1)=r/e(\mathfrak{B}_{E_1}\mid \mathfrak{o}_{E_1})$.

On a alors $\nu _{\A(E_1)}(\beta)=-n(E_1)$ et $\nu_{\B _E (E_1)} (b) = -r(E_1)$ .

D'après \cite[1.4.13]{BK}, on a: 

 $k_0(\beta , \mathfrak{A}(E_1))=k_0(\beta,\A)/e(\B_{E_1} \mid \mf{o}_{E_1}) ,$

$k_0(b,\B_E(E_1))=k_0(b,\B _E ) /e( \B _{E_1} \mid \mf{o} _{E_1} )$

Ainsi $[\A (E_1),n(E_1),r(E_1),\beta ]$ et $[\B_E(E_1),r(E_1),r(E_1)-1,b]$ sont des strates simples et vérifient les conditions de \cite[2.2.3]{BK}.

Donc $[\A(E_1),n(E_1),r(E_1)-1,\beta+b]$ est simple le corps $F[\beta +b] =F[\beta , b ]$ est égal au corps $F[\beta , b ]$ et $k_0(\beta +b , \A(E_1) )= \left\{ \begin{array}{ll}
        -r(E_1)=k_0(b,\A(E_1)$ si $b \not \in E  \\
       k_0(\beta , \A(E_1)$ si $b\in E 
    \end{array}  
\right.$

Déduisons en l'énoncé.

-$F[\beta +b]$ est un corps 

-$\nu _{\A} (\beta +b) =-n$

-d'après \cite[1.4.13]{BK} $k_0(\beta +b ,\A )=k_0 (\beta +b, \A ( E_1) ) \times e(\B _{E_1} \mid \mf{o} _{E_1}),$ donc d'après ce qui précède $k_0(\beta +b ,\A )= \left\{ \begin{array}{ll}
        
-r =k_0(b,\B _E )$ si $b \not \in E  \\
       k_0(\beta , \A )$ si $b \in E)
    \end{array}  
\right.$

donc $[\A,n,r-1,\beta +b]$ est une strate simple.

\end{proof}

\begin{rema} Dans les notations \cite[2.2.3]{BK}, si l'on supprime l'hypothèse que $F[\beta,s(b)]$ est maximal (en gardant l'hypothèse que c'est un corps), on a $F[\beta,s(b)]^{\times}\subset \mf{K} (\A)$ mais on n'a pas à priori $F[\beta , b]^{\times}\subset \mf{K} (\A)$ (en fait $F[\beta , b]$ n'est pas nécessairement un corps). Ici, puisque $E/F$ est modérée et $s(b)=b$, on a $F[\beta,b]^{\times}\subset \mf{K}(\A)$.
\end{rema}
Rappelons que $[\mathfrak{A},n,r,\beta]$ une strate pure dans l'algèbre $A=\mathrm{End}_F(V)$ tel que $F[\beta]$ est modérément ramifiée sur $F$.

 \begin{prop}  \cite[3.1]{Esse} \label{approxmodéré1} Il existe un élément $\gamma$ dans le corps $F[\beta]$ tel que la strate $[\mathfrak{A},n,r,\gamma]$ est simple et équivalente à $[\mathfrak{A},n,r,\beta]$.
 \end{prop}

 \begin{prop} \label{approxmodéré2} $\widetilde{BKH}$ 
 
 Supposons de plus que $r=-{k_0}_F(\beta,\mathfrak{A})$. Alors il existe un élément $\gamma$ dans le corps $F[\beta]$ tel que les deux conditions suivantes soient vérifiées:
  
 - la strate $[\mathfrak{A},n,r,\gamma]$ est simple et équivalente à $[\mathfrak{A},n,r,\beta]$.
 
 - la strate $[\mathfrak{B}_{\gamma},r,r-1,\beta - \gamma ]$ est simple.
 
 \end{prop}
 
 \begin{proof}  En utilisant un argument similaire à la proposition \ref{modif2.2.3}, il suffit de traiter le cas où $F[\beta]$ est un sous corps maximal de l'algèbre $A=\mathrm{End}_F(V)$. D'après \ref{approxmodéré1} il existe une strate simple [$\mathfrak{A} , n ,r ,\gamma $] telle que 
 $[\mathfrak{A} , n ,r, \beta ]\sim[\mathfrak{A} , n ,r, \gamma ]$ et $F[\gamma] \subset F[\beta]$.\\
 
Montrons que  [$\mathfrak{B}_{\gamma},r ,r-1, \beta -\gamma $] est une strate pure dans l'algèbre $\mathrm{End}_{F[\gamma]}(V)$,
 
 -$F[\gamma][\beta -\gamma ] = F[\beta]$ est un corps

 -$F[\gamma][\beta-\gamma]^{\times} \subset \mf{K}(\mf{B}_{\gamma})=\mf{K}(\mf{A})\cap B_{\gamma}^{\times}$

 - Il reste à montrer que $\nu_{\mf{B}_{\gamma}}(\beta - \gamma )=-r$.
 
 Montrons que:

$-(a)\nu _{\mf{B}_{\gamma}}(\beta -\gamma ) \geq -r $

$-(b) \nu _{\mf{B}_{\gamma}}(\beta -\gamma ) \leq -r$

$(a)[\mathfrak{A} , n ,r, \beta ]\sim[\mathfrak{A} , n ,r, \gamma ]$ et $F[\gamma] \subset F[\beta]$ donc $\beta-\gamma \in \mf{P}^{-r} \cap B_{\gamma}=\mf{Q}_{\gamma}^{-r}$, donc $\nu_{\mf{B}_{\gamma}}(\beta -\gamma) \geq  -r$

$(b)$ D'après \cite[1.4.15]{BK}, $\nu_{\mf{B}_{\gamma}}(\beta -\gamma) \leq {k_0}_{F[\gamma]}(\beta -\gamma,\mf{B}_{\gamma})$, de plus par hypothèse, ${k_0}_{F}(\beta,\mf{A})=-r.$

Il suffit donc de montrer que ${k_0}_{F[\gamma]}(\beta -\gamma ,\mf{B}_{\gamma}) \leq {k_0}_{F}(\beta ,\mf{A})$. On a \cite[1.4.5]{BK}:

${k_0}_{F[\gamma]}(\beta -\gamma ,\mf{B}_{\gamma}) =\mathrm{max}\{k\in \Z \mid \mathfrak{N} _k ( \beta - \gamma ,\mf{B}_{\gamma}) \not \subset \mf{o}_{F[\beta]} +\mf{Q}_{\gamma}\}$~~~~\footnote{Ici $\mf{B}_{\beta}=\mf{o}_{F[\beta]} $ car $F[\beta]$ est un sous corps maximal de l'algèbre $\mathrm{End}_{F[\gamma]}(V)$ \cite[1.2.2,1.2.3,1.4.5]{BK}}

${k_0}_F (\beta ,\mf{A}) = \mathrm{max}\{ k\in \Z \mid \mathfrak{N} _k(\beta ,\mf{A}) \not \subset \mf{o}_{F[\beta]} +\mf{P} \}$~~~~\footnote{Ici $\mf{B}_{\beta}=\mf{o}_{F[\beta]} $ car $F[\beta]$ est un sous corps maximal de l'algèbre $A=\mathrm{End}_{F}(V)$\cite[1.2.2,1.2.3,1.4.5]{BK}}

Il suffit donc de montrer que ${\mathfrak{N}} _{{k_0}_F(\beta,\mf{A})+1} (\beta - \gamma ,\mf{B}_{\gamma}) \subset \mf{o}_{F[\beta]} +\mf{Q}_{\gamma}$. On a \cite[1.4.3]{BK}:

${\mathfrak{N}}_{{k_0}_F(\beta ,\mf{A})+1}(\beta -\gamma , \mf{B} _{\gamma} ) =$
$\{ x\in \mf{B} _{\gamma} \mid  (\beta - \gamma )x -x(\beta - \gamma ) \in \mf{Q}_{\gamma}^{{k_0}_F(\beta,\mf{A})+1}\}=$\footnote{car $\gamma x =x \gamma,$ puisque $x\in \mf{B}_{\gamma}\subset B_{\gamma}$}$\{ x\in \mf{B} _{\gamma} \mid  \beta  x -x\beta  \in \mf{Q}_{\gamma}^{{k_0}_F(\beta,\mf{A})+1}\}\subset  \mathfrak{N} _{ {k_0}_F (\beta, \mf{A})+1}(\beta ,\mf{A} ) \cap \mf{B}_{\gamma}$
$ \subset$ $(\mf{o} _{F[\beta]} +\mf{P}) \cap \mf{B}_{\gamma} $
$= \mf{o} _{F[\beta]} +\mf{Q} _{\gamma}$

  cela conclut (b)\\
 
 Montrons que [$\mathfrak{B}_{\gamma},r ,r-1, \beta -\gamma $] est simple, puisqu'elle est pure et que $F[\beta -\gamma]$ est modérément ramifiée sur $F[\gamma]$, d'après \ref{approxmodéré1} il existe une strate simple  [$\mathfrak{B}_{\gamma},r ,r-1, \alpha$] tel que 
  $[\mathfrak{B}_{\gamma},r ,r-1, \alpha ] \sim [\mathfrak{B}_{\gamma},r ,r-1, \beta -\gamma ]$ et $F[\gamma][\alpha] \subset F[\gamma][\beta- \gamma ]$.
  D'après \ref{modif2.2.3}, $[\mathfrak{A},n,r-1,\gamma+\alpha ]$ est simple et $F[\gamma +\alpha]=F[\alpha,\gamma]$.

   On a $\alpha \equiv \beta -\gamma \pmod{ \mathfrak{Q}_{\gamma}^{-(r-1)}} $ et donc 
  $\gamma +\alpha \equiv  \beta  \pmod{  \mathfrak{P}^{-(r-1)}}$, ainsi  \begin{center}
  $[\mathfrak{A},n,r-1,\gamma+\alpha ]\sim [\mathfrak{A},n,r-1,\beta ].$\end{center}
  
  Or  $[\mathfrak{A},n,r-1,\gamma+\alpha ]$ et $[\mathfrak{A},n,r-1,\beta ]$ sont simples, la première par construction et la deuxième par hypothèse $({k_0}_F(\beta,\mf{A})=-r)$.
  
  Donc d'après \cite[2.4.1 (i)]{BK}, $[F[\gamma + \alpha ]:F] =[F[\beta]:F]$ et ainsi  $F[\gamma,\alpha]=F[\gamma + \alpha ]=F[\beta]$.
  
  Donc 
  
 - $[\mathfrak{B}_{\gamma},r ,r-1, \alpha ]$ est une strate simple dans l'algèbre $\mathrm{End}_{F[\gamma]}(V)$
 
 -  $F[\gamma][\alpha]$ et un sous corps maximal de la  $F[\gamma]$- algèbre $\mathrm{End}_{F[\gamma]}(V)$

 -$[\mathfrak{B}_{\gamma},r ,r-1, \alpha ] \sim [\mathfrak{B}_{\gamma},r ,r-1, \beta -\gamma ]$

 par conséquent, d'après \cite[proposition 2.2.2]{BK}, $[\mathfrak{B}_{\gamma},r ,r-1, \beta -\gamma ]$ est simple.
 
 \end{proof}

 \section{Élément minimaux et représentant standard }\label{emrs}
 
 Rappelons que $F$ désigne un corps local non archimédien.
 
 \begin{prop} \cite[Chapter II Proposition 5.7]{Neuk}
 
 Soit $K$ une extension finie de $F$, soit $q=p^f$ le cardinal du corps résiduel de $K$ et soit 
  $\mu _{q-1}$  les racines $(q-1)-$ième de l'unité dans $K$, alors:
  
 \begin{enumerate}
 \item[($i$)] Si $K$ est de caractéristique $0$, on a des isomorphismes de groupes topologiques:
 \begin{center}
 $K^{\times} \simeq \Z \times \mf{o}_K ^{\times} \simeq \Z \times \mu _{q-1} \times 1+ \mf{p} _K \simeq \Z \times \Z / (q-1) \Z \times \Z /p^a \Z \times \zp ^d $
 \end{center}
 
 où $a\geq 0$ est un entier et $d=[K:\qp]$
 
 \item[($ii$)] Si $K$ est de caractéristique $p$, on a des isomorphismes de groupes topologiques:
 \begin{center}
 $K^{\times} \simeq   \Z \times \mf{o}_K ^{\times}\simeq \Z \times \mu _{q-1} \times 1+ \mf{p} _K  \simeq \Z \times  \Z /(q-1)\Z \times \zp ^{\N}$ \end{center}
 \end{enumerate}
 \end{prop}

\begin{coro} Soit $E$ une extension finie modérément ramifiée de $F$. Soit $\pi_F$ une uniformisante de $F$. Il existe une uniformisante $\pi _E$ de $E$ et une racine de l'unité $z$, d'ordre premier à $p$, tel que $\pi _ E ^e z = \pi _F$.
 
\end{coro}

Soit $E/F$, $\pi_F$ et $\pi_E$ comme dans le corollaire précédent. On définit un sous groupe $C_E:=<\pi_E, \mu _{q-1} >$ de $E^{\times}$. Pour tout élément $c$ de $E^\times$ il existe un unique élément $sr(c) \in C_E$ tel que $c=sr(c)\times x $. L'élément $sr(c)$ est appelé représentant standard de $c$. L'élément $sr(c)$ est l'unique élément dans $C_E$ tel que $\nu _ E (sr(c)- c) >\nu _E (c) $.
  
 \begin{prop}\label{CE}\begin{enumerate} Soit \begin{small} \shorthandoff{;:!?} \xymatrix @!=0,01cm{E' \ar @{-}[d]\\ E\ar @{-}[d] \\ F } \end{small} une tour d'extensions finies modérément ramifiées de corps. Alors
 \item[($i$)]$C_E \subset C_{E'} $.
 \item[($ii$)]Si $E/F$ est une extension galoisienne, alors $C_E$ est stable par l'action du groupe de Galois $\mathrm{Gal}(E/F)$ sur $E$. De plus si $\sigma _ 1 , \sigma _2 \in \mathrm{Gal}(E/F)$ et $s\in C_E$ sont tel que $\sigma_1 (s) \not = \sigma _2 (s)$, alors $\nu_E (\sigma _ 1(s) - \sigma _2(s))= \nu _E (\sigma _1 (s))= \nu _E (s)$ 
   \end{enumerate}
  \end{prop}
 \begin{proof}
 ($i$) est immédiat. La première assertion de ($ii$) est simple, la deuxième résulte de la première et du fait que puisque $\sigma _1(s)$ est dans $C_E$ alors $sr(\sigma _1 (s))=\sigma _1 (s)$ donc $\nu_E (\sigma _ 1(s) - \sigma _2(s))= \nu _E (s)$.
 \end{proof}
  \begin{prop}\label{minigene}
Soit $E/F$ modérément ramifiée, soit $\beta\in E$, les assertions  suivantes sont équivalentes:

\begin{enumerate}

\item[($i$)] L'élément $\beta $ est minimal relativement à l'extension $E/F$.

\item[($ii$)] Le représentant standard de $\beta$  engendre l'extension $E/F$, i.e. $F[sr(\beta)]=E$.

 \end{enumerate}
 
 \end{prop}

\begin{proof} $\widetilde{BKH}$

($i$) implique ($ii$)

Posons $\nu=-\nu_E(\beta)$, $ e =e(E,F)$.

Notons $E'$ l'extension non ramifiée maximale contenue dans $E$.

Il suffit de montrer que $E'\subset F[sr(\beta)]$ et $E\subset E'[sr(\beta)]$\\

D'une part, par définition, $\nu_E(sr(\pi_F^{-\nu}\beta^e)-\pi_F^{-\nu}\beta ^e)>0$,et donc $sr(\pi_F^{-\nu}\beta^e)+\mathfrak{p}_E=\pi_F^{-\nu}\beta ^e+\mathfrak{p}_E.$ D'autre part, $sr(\pi_F^{-\nu}\beta ^e)=\pi_F^{-\nu}sr(\beta)^e.$
 
 On déduit que $\pi_F^{-\nu}sr(\beta)^e +\mathfrak{p}_E$ engendre $k_E / k_F$.
 
 Donc $\pi_F^{-\nu}sr(\beta)^e$ génère $E'$. 
 
 Donc $E'\subset F[sr(\beta)]$.\\
 
$\nu_E(\beta)=\nu_E(sr(\beta)),$ donc $\mathrm{pgcd}(\nu_E(sr(\beta)),e)=1.$
 
 Soient $a,b$ entiers vérifiant $a\nu_E(sr(\beta))+be=1$.
 
 $\nu_E(sr(\beta)^a \pi_F ^b)=1$ donc  $E'[sr(\beta)^a \pi_F ^b]=E$ \footnote{Une extension finie totalement ramifiée  est générée par une uniformisante quelconque.}
 
 Donc $E'[sr(\beta)]=E$.
 
Cela conclut ($i$) implique ($ii$).

($ii$) implique ($i$)

-$E=F[sr(\beta)]\subset F[\beta]\subset E$ donc $E=F[\beta]$.

Posons $\nu= \nu _E ( sr(\beta)) =\nu_E(\beta)$ et $e=e(E,F)$

\cite{wald}$E'$ est engendrée sur $F$ par les racines de l'unité d'ordre premier à $p$ contenues dans $E$. Soient $d=pgcd(\nu ,e)$ et $b=e/d$. On voit que $sr(\beta)^b\in E'$. Comme $sr(\beta)$ engendre $E$ sur $F$ et donc sur $E'$ on en déduit $[E:E']
\leq b$. Or $[E:E']=e$ donc $b=e$ et $\nu$ est premier à $e$.

De plus $sr(\beta)$ engendre $E'$ sur $F$, ou encore $\pi_F ^{-\nu} sr(\beta)^e$ engendre $E'$ sur $E$. Or ce terme est une racine de l'unité d'ordre premier à $p$ donc engendre $E'$ si et seulement si sa réduction engendre $k_{E'}=k_E$.

Or $\pi_F^{-\nu} sr(\beta)^e +\mathfrak{p}_E = \pi_F ^{-\nu} \beta ^{e} +\mathfrak{p}_E$. Donc $\pi_F ^{-\nu} \beta ^{e} +\mathfrak{p}_E$ engendre $k_E$.

Cela conclut ($ii$) implique ($i$).

 \end{proof}
 
 \section{Elément générique et caractère générique au sens de  Yu\cite[§8,9]{Yu}}\label{egcg}

 Soit $G'\subset G$ une suite de Levi tordue modérée.
 
 On note $Z'$ le centre de $G'$, soit $T$ un tore maximal de $G$. On plonge $\mathrm{Lie} ^* (Z')^{\circ}$  canoniquement dans $\mathrm{Lie}^*(G')$. On plonge aussi  $\mathrm{Lie}^*(G')$ dans $\mathrm{Lie}^*(G)$.
 
 Un élément $X^* \in \mathrm{Lie}^*(Z')^{\circ}$ est dit $G$-générique de profondeur $\mathbf{r}$ si les conditions suivantes sont satisfaites:
 \begin{enumerate}
 
 \item[(GE1)] : $\mathrm{ord}(X^*(H_a))=-\mathbf{r}$ pour tout $a\in \phi(G,T,\bar{F}) \backslash \phi(G',T,\bar{F}).$

 \item[(GE2)] : $Z_{W}(\tilde{X}^*)=W(\phi(G',T,\bar{F}))$ où:
 
 \end{enumerate}
On a un isomorphisme $\mathrm{Lie}^*(T) \otimes _{\Z}\bar{F} \simeq X^*(T) \otimes _{\Z} \bar{F}$, soit $w_\mathbf{r}$ tel que $w_{\mathbf{r}} X^*$ soit de profondeur $0$.

On identifie $w_{\mathbf{r}} X^*$ et son image, via l'isomorphisme ci dessus, qui appartient à $X^*(T) \otimes _{\Z} \mf{o}_{\bar{F}}$.
 
L'élément $\tilde{X}^*$ est alors  réduction modulo $X^*(T) \otimes _{\Z} \mf{p}_{\bar{F}}$ de $w_r X^*$.

  $Z_W(\tilde{X}^*)$ désigne le stabilisateur dans $W=W(\phi(G,T,\bar{F}))$ de $\tilde{X}^*$ par l'action de $W$ sur $X^*(T) \otimes _{\Z} \mf{o}_{\bar{k}}$.\\

 Un caractère $\phi$ de $G'(F)$ est dit $G-$générique (relativement à $y$ ) de profondeur $\mathbf{r}$ si $\phi $ est trivial sur $G'(F)_{y, \mathbf{r}+}$, non trivial sur $G'(F)_{y,\mathbf{r}}$ et $\phi $ restreint à $G'(F)_{y,\mathbf{r}:\mathbf{r}+}$ est réalisé par un élément $X^*\in (\mathrm{Lie}^*(Z')^{\circ})_{-\mathbf{r}} \subset (\mathrm{Lie}^*(G')_{y,-\mathbf{r}}$ qui est $G$-générique de profondeur $\mathbf{r}$.
 
 \section{Liens entre caractères simples modérés et caractères génériques}\label{cscg}
 
Soit $[\A,n,0,\beta]$ une strate simple dans l'algèbre $A=\mathrm{End}_F(V)$ tel que $F[   \beta]$ est une extension modérément ramifiée de $F$.
D'après la proposition \ref{approxmodéré2} on en déduit \cite[2.4.2]{BK} qu'on peut trouver une suite d'approximation  $[\A,n,r_i,\beta_i]$, $0 \leq i \leq s $ tel que :
\begin{enumerate}

\item[($vii$)]$F[\beta _{i+1}]\subset F[\beta _i ] $ pour $0\leq i \leq s-1$ ($\beta _0 =\beta $).

\item[($vi'$)] La strate $[\A _{i+1}, r_{i+1},r_{i+1}-1,\beta_i - \beta _{i+1} ]$ est simple pour $0\leq i \leq s-1$.
\end{enumerate}
   Où $\A _i$ est défini par $\A_i = \A \cap \mathrm{End}_{F[\beta _i]}(V)$.\\

   Posons alors:
   \begin{enumerate}
   
   \item[$\bullet$] $c_i=\beta _ i - \beta _{i+1} $ si $0\leq i \leq s-1 $
   
  \item[$\bullet$] $c_s = \beta _s$
   
   \item[$\bullet$]$E_i=F[\beta _i]$ pour $0\leq i \leq s$ et $E_d=F$.
   
   \item[$\bullet$] $A_i=\mathrm{End}_{E_i}(V)$, $\A _ i = A_i \cap \A$ et $\p _ i = \p  \cap \A _i$.
   \item[$\bullet$] $N_i = \mathrm{dim}_{E_i}(V)=\frac{N}{[E_i : F]}$ et $D_i= [E_i : F]= \frac{N}{N_i}$
   \end{enumerate}
    \begin{prop} On a une égalité de corps $E_i=E_{i+1}[c_i]$ pour $0\leq i \leq s$, de plus $c_i$ est minimal sur $E_{i+1}$.
   
   \end{prop}
   \begin{proof}Cela découle de ($vi'$) (cf \ref{approxmodéré2}) et \ref{minider}.
   \end{proof}

    Distinguons deux cas: 
    (Cas A) \underline{$\beta_s $ appartient à $F$}
   et (Cas B) \underline{ $\beta _s $ n'appartient pas à $F$}.

   \begin{enumerate}
   \item[(Cas A)]

   \underline{Si $\beta_s $ appartient à $F$} on pose $d=s$.\\
   
   On a dans ce cas $H^1(\beta , \A ) = (1+\p _0 ^1) (1+\p _1 ^{[\frac{r_1}{2}]+1})\cdots (1+ \p ^{[\frac{r_d}{2}]+1}).$\\
   
   Posons $G^i = \mathrm{Res} _{E_i /F} \mathrm{Aut}_{E_i} (V)$ pour $0 \leq i \leq d=s$,  ce sont des schémas en groupes réductifs.\\

    On a des plongements naturels $G^0 \hookrightarrow G^1 \hookrightarrow \cdots \hookrightarrow G^i \hookrightarrow \cdots \hookrightarrow G^d $ de groupes algébriques.

   \begin{prop} La suite $G^0 \hookrightarrow G^1 \hookrightarrow \cdots \hookrightarrow G^i \hookrightarrow \cdots \hookrightarrow G^d $  est une suite de Levi tordue modérée tel que $Z(G^0)/Z(G)$ est anisotrope.

\end{prop}

\begin{proof} $G^i \simeq \mathrm{Res} _{E_i /F } \mathrm{GL}_{N_i}$, $E_i$ est une extension modérément ramifiée de $F$, soit $\bar{E_i} ^{Gal} $ la clôture galoisienne de $E_i$, alors $\bar{E_i} ^{Gal} $ est une extension modérément ramifiée de $F$.
Soit $P$ le polynôme minimal de $\beta _i$ sur $F$, on a alors \begin{center} $E_i \otimes _ F \bar{E_i} ^{Gal} \simeq F[\beta _i] \otimes _F \bar{E_i} ^{Gal} \simeq F[X]/P(X) \otimes _F \bar{E_i} ^{Gal} \simeq \bar{E_i} ^{Gal} [X]/P(X) \simeq (\bar{E_i} ^{Gal})^{[E_i :F]}$.\end{center}

Donc $G^i(\bar{E_i} ^{Gal})\simeq \mathrm{GL}_{N_i} ( E_i \otimes _F \bar{E_i} ^{Gal} ) \simeq \mathrm{GL}_{N_i} ( (\bar{E_i} ^{Gal}) ^{[E_i :F]})\simeq (\mathrm{GL}_{N_i} ( \bar{E_i} ^{Gal}) )^{[E_i :F]}$.\\

 Ainsi $G^i \otimes \bar{E_i} ^{Gal} $ est un Levi de $G\otimes \bar{E_i} ^{Gal}$.\\

Puisque $Z(G^0)(F) /Z (G) (F) \simeq E_0 ^{\times} / F^{\times} $ est compact,  $Z(G^0)/Z(G)$ est anisotrope.

\end{proof}

 \begin{prop}\label{filtA} Il existe des applications $\mathcal{B}(G^0,F) \hookrightarrow  \cdots \hookrightarrow \mathcal{B}(G^i,F) \hookrightarrow \cdots \mathcal{B}(G^d,F)$, des\footnote{On a mis en gras les nombres réels $\mathbf{r}$ " du coté Moy-Prasad" pour les différencier de ceux " du coté Bushnell-Butzko".} nombres réels $\mathbf{r}_i$, et un point $y\in \mathcal{B}(G^0 , F)$ tel que pour tout $1\leq i \leq d $:
 
 \begin{enumerate}
 \item[$\bullet$] $ 1+ \p _i ^{[\frac{r_i}{2}]+1} = G^i(F)_{y, \frac{\mathbf{r}_{i-1}}{2}+}$
 \item[$\bullet$] $ 1+ \p _i ^{[\frac{r_i+1}{2}]} = G^i(F)_{y, \frac{\mathbf{r}_{i-1}}{2}}$
 \item[$\bullet$] $ 1+ \p _i ^{r_i+1} = G^i(F)_{y, \mathbf{r}_{i-1} +}$
 \item[$\bullet$] $ 1+ \p _i ^{r_i} = G^i(F)_{y, \mathbf{r}_{i-1}} $
 \item[$\bullet$] $  \p _i ^{[\frac{r_i}{2}]+1} = \mathfrak{g}^i(F)_{y, \frac{\mathbf{r}_{i-1}}{2}+}$
 \item[$\bullet$] $  \p _i ^{[\frac{r_i+1}{2}]} = \mathfrak{g}^i(F)_{y, \frac{\mathbf{r}_{i-1}}{2}}$
 \item[$\bullet$] $  \p _i ^{r_i+1} = \mathfrak{g}^i(F)_{y, \mathbf{r}_{i-1} +}$
 \item[$\bullet$] $  \p _i ^{r_i} = \mathfrak{g}^i(F)_{y, \mathbf{r}_{i-1}} $ 
 \end{enumerate}

\end{prop} 

\begin{proof}
   
   C'est l'objet de \cite{Brou}, de l'appendice A de \cite{Brou}, voir aussi \cite[2.4]{PY}    \end{proof}

   On a $\mathbf{r}_i =- \mathrm{ord}(c_i)$ pour $0\leq i \leq d-1$. En effet, posons $e=e(\A \mid\mathfrak{o}_F) ,$
   
    alors  $e \mathbf{r}_{i-1}= r_i  $ d'après \cite[2.4]{PY} pour $1\leq i \leq d $
    
     et $r_{i+1}=-\nu_{\A _{i+1}}(c_i)=-\nu_{\A }(c_i)=-\frac{\nu _{E_i} (c_i) e(\A \mid \mathfrak{o}_F)}{e(E_i \mid F)}=-\mathrm{ord}(c_i)e(\A \mid \mathfrak{o} _F)=-\mathrm{ord}(c_i)e $ pour $0\leq i \leq d-1$.\\
   
   Posons $\mathbf{r}_d = -\mathrm{ord}(\beta _s).$

   \begin{prop} \label{factA} Soit $\theta \in \mathcal{C}(\A,m,\beta)$, posons $m_i=\mathrm{max}\{m,[\frac{r_i}{2}]\}$ pour $0\leq i \leq s=d$ il existe $\phi_0,\ldots,\phi_s$  des caractères de $E_0 ^{\times},\ldots,E_s ^{\times}$ tel que:
   
   \begin{center}
   
  $ \theta= \prod\limits_{i=0} ^{(s=d)} \theta ^i $
   
   \end{center}
   
   où $\theta_i$ est défini par les $2$ conditions:
   
 $\bullet$  $\theta ^i \mid _{(1+\p _0^{m_0+1})\ldots (1+ \p _i ^{m_i+1})}(1+x)=\phi _i \circ \mathrm{det}_{A_i}(1+x)$

  $\bullet$ $\theta  ^i \mid_{(1+ \p_{i+1}^{m_{i+1}+1})\cdots (1+ \p ^{m_d +1 })} (1+x) =\psi \circ \mathrm{tr}_{A/F}(c_i x)$ , si $i \leq d-1$.

   \end{prop}

   \begin{proof} $\widetilde{BKH}$ Montrons le résultat par récurrence sur la longueur $s$ de la suite d'approximation. Si $s=0$, alors $\beta = \beta _0$ est minimal et la proposition résulte de la définition \ref{defsimpl}.
   Supposons le résultat vrai au rang $s-1$.
   
   D'après \ref{defsimpl}, $\theta \mid _{H^{m_1+1}}= \psi _{c_0} \times \theta ' $ avec $\theta ' \in \mathcal{C}(\A,m_1,\beta _1 )$
   D'après l'hypothèse de récurrence il existe $\phi _1,\ldots,\phi _s$ tel que $\theta '= \prod\limits_{i=1}^s {\theta '} ^i $ où :
  
   $\bullet$ $ {\theta ' }^i \mid _{(1+\p_1 ^{m_1 +1}) \cdots (1+\p_i ^{m_i +1}) }(1+x)=\phi _i\circ det_{A_i}(1+x)$
   
   $\bullet $ ${\theta '} ^i \mid _{(1+ \p _{i+1}^{m_{i+1}+1})\cdots(1+\p  ^{m_d +1})}(1+x) = \psi \circ \mathrm{tr}_{A/F}(c_i x)$, si $i \leq d-1$.

   On prolonge ${\theta '} ^i$ à $H^m(\beta_0, \A)$ grâce à $\phi _ i $ et on note $\theta ^i$ ce prolongement:
   
  $\bullet$ $\theta ^i \mid _{(1+\p _0^{m_0+1})\cdots(1+ \p _i ^{m_i+1})}(1+x)=\phi _i \circ det_{A_i}(1+x)$

   $\bullet$ $\theta  ^i \mid_{(1+ \p_{i+1}^{m_{i+1}+1})\cdots(1+ \p ^{m_d +1 })} (1+x) ={\theta '} ^i \mid_{(1+ \p_{i+1}^{m_{i+1}+1})\cdots(1+ \p ^{m_d +1 })} (1+x)=\psi \circ \mathrm{tr}_{A/F}(c_i x)$

   Posons $\theta ^0  = \theta \times  \prod\limits_{i=1}^s (\theta ^i )^{-1}$.

   Alors $\theta ^0 \mid _{(1 +\p _0    ^{m_0 +1})}$ factorise par $det_{A_0}$. Il existe donc un caractère lisse $\phi _0$ de $E_0 ^{\times}$ tel que $\theta ^0 \mid _{(1 +\p _0    ^{m_0 +1})}(1+x)=\phi _0 \circ det_{A_0} (1+x)$.
   
   Enfin \begin{align*} \theta ^0 \mid_{(1+ \p_{1}^{m_{1}+1})\cdots(1+ \p ^{m_d +1 })}(1+x)=& \left( \theta  \mid_{(1+ \p_{1}^{m_{1}+1})\cdots(1+ \p ^{m_d +1 })} \times  \prod\limits_{i=1}^s (\theta ^i )^{-1} \mid_{(1+ \p_{1}^{m_{1}+1})\cdots(1+ \p ^{m_d +1 })}\right)(1+x)\\
    &= \left( \theta \mid _{H^{m_1+1}} \times (\theta ')^{-1}\right)(1+x)\\
    &= \left( \psi _{c_0} \times \theta ' \times (\theta ')^{-1}\right)(1+x)\\
    & = \psi _{c_0}(1+x) \\
    &=\psi \circ \mathrm{tr}_{A/F}(c_0 x) 
   \end{align*}

   \end{proof}
   
    \item[(Cas B)]

   \underline{Si $\beta_s $ n'appartient pas à $F$} on pose $d=s+1$.\\
   
   On a dans ce cas $H^1(\beta , \A ) = (1+\p _0 ^1) (1+\p _1 ^{[\frac{r_1}{2}]+1})\cdots (1+ \p ^{[\frac{n}{2}]+1}).$\\
   
   Posons $G^i = \mathrm{Res} _{E_i /F} \mathrm{Aut}(_{E_i} (V)$ pour $0 \leq i \leq s=d-1$, et posons $G^d=G=\mathrm{Aut}_F(V) $ ce sont des schémas en groupes réductifs.\\

    On a des plongements naturels $G^0 \hookrightarrow G^1 \hookrightarrow \cdots \hookrightarrow G^i \hookrightarrow \cdots \hookrightarrow G^d $ de groupes algébriques.

   \begin{prop} La suite $G^0 \hookrightarrow G^1 \hookrightarrow \cdots \hookrightarrow G^i \hookrightarrow \cdots \hookrightarrow G^d $  est une suite de Levi tordue modérée tel que $Z(G^0)/Z(G)$ est anisotrope.

\end{prop}

\begin{proof} Même preuve que dans le (Cas A).

\end{proof}

 \begin{prop}\label{filtB} Il existe des applications $\mathcal{B}(G^0,F) \hookrightarrow  \cdots \hookrightarrow \mathcal{B}(G^i,F) \hookrightarrow \cdots \mathcal{B}(G^d,F)$, des\footnote{On a mis en gras les nombres réels $\mathbf{r}$ " du coté Moy-Prasad" pour les différencier de ceux " du coté Bushnell-Butzko".} nombres réels $\mathbf{r}_i$, et un point $y\in \mathcal{B}(G^0 , F)$ tel que pour tout $1\leq i \leq s=d-1 $:
 
 \begin{enumerate}
 \item[$\bullet$] $ 1+ \p _i ^{[\frac{r_i}{2}]+1} = G^i(F)_{y, \frac{\mathbf{r}_{i-1}}{2}+}$
 \item[$\bullet$] $ 1+ \p _i ^{[\frac{r_i+1}{2}]} = G^i(F)_{y, \frac{\mathbf{r}_{i-1}}{2}}$
 \item[$\bullet$] $ 1+ \p _i ^{r_i+1} = G^i(F)_{y, \mathbf{r}_{i-1} +}$
 \item[$\bullet$] $ 1+ \p _i ^{r_i} = G^i(F)_{y, \mathbf{r}_{i-1}} $
 \item[$\bullet$] $  \p _i ^{[\frac{r_i}{2}]+1} = \mathfrak{g}^i(F)_{y, \frac{\mathbf{r}_{i-1}}{2}+}$
 \item[$\bullet$] $  \p _i ^{[\frac{r_i+1}{2}]} = \mathfrak{g}^i(F)_{y, \frac{\mathbf{r}_{i-1}}{2}}$
 \item[$\bullet$] $  \p _i ^{r_i+1} = \mathfrak{g}^i(F)_{y, \mathbf{r}_{i-1} +}$
 \item[$\bullet$] $  \p _i ^{r_i} = \mathfrak{g}^i(F)_{y, \mathbf{r}_{i-1}} $ 
 \end{enumerate}

\end{prop} 
   
   On a $\mathbf{r}_i =- \mathrm{ord}(c_i)$ pour $0\leq i \leq s-1=d-2$. \\
   
   Posons $\mathbf{r}_d = \mathbf{r}_{d-1}=  -\mathrm{ord}(\beta _s).$
   
   On a alors $1+ \p _d ^{n+1} = G^d_{y,\mathbf{r_{d-1}} +}$

   \begin{prop} \label{factB} Soit $\theta \in \mathcal{C}(\A,m,\beta)$, posons $m_i=\mathrm{max}\{m,[\frac{r_i}{2}]\}$ pour $0\leq i \leq s$, $m_d =\mathrm{max}\{m,\frac{n}{2} \}$ il existe $\phi_0,\ldots,\phi_s$  des caractères de $E_0 ^{\times},\ldots,E_s ^{\times}$ tel que:
   
   \begin{center}
   
  $ \theta= \prod\limits_{i=0} ^{(s=d-1)} \theta ^i $
   
   \end{center}
   
   où $\theta_i$ est défini par les $2$ conditions:
   
 $\bullet$  $\theta ^i \mid _{(1+\p _0^{m_0+1})\ldots (1+ \p _i ^{m_i+1})}(1+x)=\phi _i \circ \mathrm{det}_{A_i}(1+x)$

  $\bullet$ $\theta  ^i \mid_{(1+ \p_{i+1}^{m_{i+1}+1})\cdots (1+ \p ^{m_d +1 })} (1+x) =\psi \circ \mathrm{tr}_{A/F}(c_i x)$ , si $i \leq d-1$.

   \end{prop}

   \begin{proof} Analogue au (Cas A).
   
   \end{proof}
   
   \end{enumerate}

\begin{rema} Si l'on enlève l'hypothèse que la strate simple est modérée, alors il n'est pas vrai en général que l'on puisse trouver une suite d'approximation dont la suite des corps associés est décroissante, on pourrait cependant définir des groupes algébriques qui serait encore des Levi de $G$, mais la suite des groupes obtenue ne serait pas croissante, ainsi bien-sûr, la construction de Yu ne couvre pas la construction de Bushnell-Kutzko. Cependant lorsque $p \not ~\mid N$, par définition, toutes les strates simples sont modérées.
\end{rema}

   \begin{prop} \label{cage} Soit $i\in \{0,\ldots ,d-1\}$, alors $\phi _ i \circ det : G^i \to \C ^{\times}$ est $G^{i+1}$- générique de profondeur $\mathbf{r}_i$
   \end{prop}
   
   \begin{proof}$\widetilde{BKHY}$
   
    On doit d'après le §\ref{egcg} vérifier (GE1) et (GE2).\\
    
     Commençons par (GE1).

   On a (\cite[2.3]{Kim}) $G^i(F)_{y,\mathbf{r}_i}=G^{i+1}(F)_{y,\mathbf{r}_i} \cap G^i(F)=\left(1+ \p _{i+1}^{r_{i +1}} \right) \cap  A_i ^{\times} = 1+ \p _i ^{r_{i +1}}$.
   
   On en déduit donc que 
   $\phi _i \circ det  \mid_{G^i(F)_{y,\mathbf{r}_i}} (1+x)= \phi _i \circ det \mid _{1 + \p _i ^{r_{i +1}}} (1+x) = \theta ^i \mid _{1+ \p _i ^{r_{i +1} } }(1+x) = \psi \circ \mathrm{tr} (c_{i} x) = \psi \circ \mathrm{tr} (sr(c_{i}) x )$. La dernière égalité est du au fait que \begin{center}$\theta ^i \mid _{1+ \p _i ^{r_{i +1}}} \in \left((1+ \p _i ^{r_{i+1}})/(1+\p _i ^{r_{_i+1}+1} \right) ^{\vee} \simeq \p _i ^{-r_{i+1}} / \p _i ^{-r_{i+1}+1}$\end{center} et que $c_i\equiv sr(c_i) \pmod{\p _i ^{-r_{i+1}+1}}$ . On a $\mathrm{ord}(c_i)=-\mathbf{r}_i$.
    Il suffit donc de montrer le lemme suivant:
    
    \begin{lemm} Soit 
 \begin{small} \shorthandoff{;:!?} \xymatrix @!=0,01cm{E' \ar @{-}[d]^{\not =}\\ E\ar @{-}[d] \\ F } \end{small} une tour d'extensions finies modérément ramifiées. Soit $V$ un $E'-$espace vectoriel de dimension finie. Soit $c\in E'$ minimal sur $E$ et $-\mathbf{r}=\mathrm{ord}(c)$.
    
    Posons $G'=\mathrm{Res}_{E'/F}\mathrm{Aut}_{E'}(V)$, $G=\mathrm{Res}_{E/F}\mathrm{Aut}_{E}(V)$ et $D=[E:F]$, $D'=[E':F].$\\
    
    Alors l'élément $x\mapsto \mathrm{tr} (sr(c) x)$ est $G-$générique de profondeur $\mathbf{r}$.
    \end{lemm} 
    \begin{proof} $\widetilde{BKHY}$\\
    
    Le choix d'une $E'-$base de $V$ donne naissance à un tore $T$ de $G'$:
    \begin{center} $T=\mathrm{Res}_{E'/F}(\mathbb{G}_m^{N'})\subset G'=\mathrm{Res}_{E'/F}\mathrm{Aut}_{E'}(V)\subset G=\mathrm{Res}_{E/F}\mathrm{Aut}_{E}(V) $.\end{center}
    
 On a un isomorphisme de $F-$algèbre   $E'\otimes \bar{F} \simeq \prod\limits_{\tau : E \to \bar{F}} \prod\limits_{ \underset{\sigma \mid _E = \tau} {\sigma : E' \to \bar{F}}}
\bar{F}$ ($\tau$ et $\sigma$ sont des  $F$-plongements)

Numérotons les plongements comme suit : Il y a $D$ plongements $\tau$, on les note $\{ \tau_1,\ldots,\tau _k,\ldots ,\tau _D\}$ pour chaque $\tau_k$ il y a $D'/D$ plongements $\sigma$, on les note $\{ \sigma_{k1},\ldots ,\sigma_{kk'},\ldots ,\sigma_{kD/D'} \}$

Sur $\bar{F}$ on obtient le diagramme suivant:

\xymatrix{&T(\bar{F})=((E'\otimes \bar{F})^{\times})^{N'} \ar[d]^*[@]{\cong} \ar@{^{(}->}[r] & G'(\bar{F})=\mathrm{GL}_{N'}(E' \otimes \bar{F})\ar[d]^*[@]{\cong} \ar@{^{(}->}[r]& G(\bar{F})=\mathrm{GL}_N(E\otimes \bar{F})\ar[d]^*[@]{\cong}\\
&(\prod\limits_{k=1}^{D} \prod\limits_{k'=1}^{D'/D}(\bar{F}^{\times} ))^{N'} \ar[d]^*[@]{\cong} \ar@{^{(}->}[r] & \mathrm{GL}_{N'}(\prod\limits_{k=1}^{D} \prod\limits_{k'=1}^{D'/D}\bar{F})\ar[d]^*[@]{\cong} \ar@{^{(}->}[r] & \mathrm{GL}_N(\prod\limits_{k=1}^{D}\bar{F})\ar[d]^*[@]{\cong}\\
&\prod\limits_{k=1}^{D} \prod\limits_{k'=1}^{D'/D}(\bar{F}^{\times N'})    \ar@{^{(}->}[r] & \prod\limits_{k=1}^{D} \prod\limits_{k'=1}^{D'/D}\mathrm{GL}_{N'}(\bar{F}) \ar@{^{(}->}[r] & \prod\limits_{k=1}^{D}(\mathrm{GL}_N(\bar{F}))}

Au niveau des algèbres de Lie on a :

\xymatrix{E' \otimes F \ar[d]^*[@]{\cong} \\ Z(\mathrm{Lie}(G'(\bar{F}))) \ar[d]^*[@]{\cong} \ar@{^{(}->}[r]&\mathrm{Lie}(T(\bar{F))}\ar[d]^*[@]{\cong}\ar@{^{(}->}[r] & \mathrm{Lie} (G'(\bar{F}))\simeq M_{N'} (E' \otimes \bar{F}) \ar[d]^*[@]{\cong} \ar@{^{(}->}[r] &  \mathrm{Lie} (G(F))\simeq \mathrm{M}_N(E\otimes \bar{F}) \ar[d]^*[@]{\cong} \\
 \prod\limits_{k=1}^{D} \prod\limits_{k'=1}^{D'/D}(\bar{F})\ar@{^{(}->}[r]& \prod\limits_{k=1}^{D} \prod\limits_{k'=1}^{D'/D}(\bar{F}^{ N'})  \ar@{^{(}->}[r] &\prod\limits_{k=1}^{D} \prod\limits_{k'=1}^{D'/D}M_{N'}(\bar{F})\ar@{^{(}->}[r] & \prod\limits_{k=1}^{D}(\mathrm{M}_N(\bar{F}))}

Notons $\iota$ l'application $E'\otimes \bar{F} \to \prod\limits_{k=1}^{D}(\mathrm{M}_N(\bar{F})) $ obtenue ci-dessus.

L'élément $sr(c)\otimes 1 $ est envoyé sur

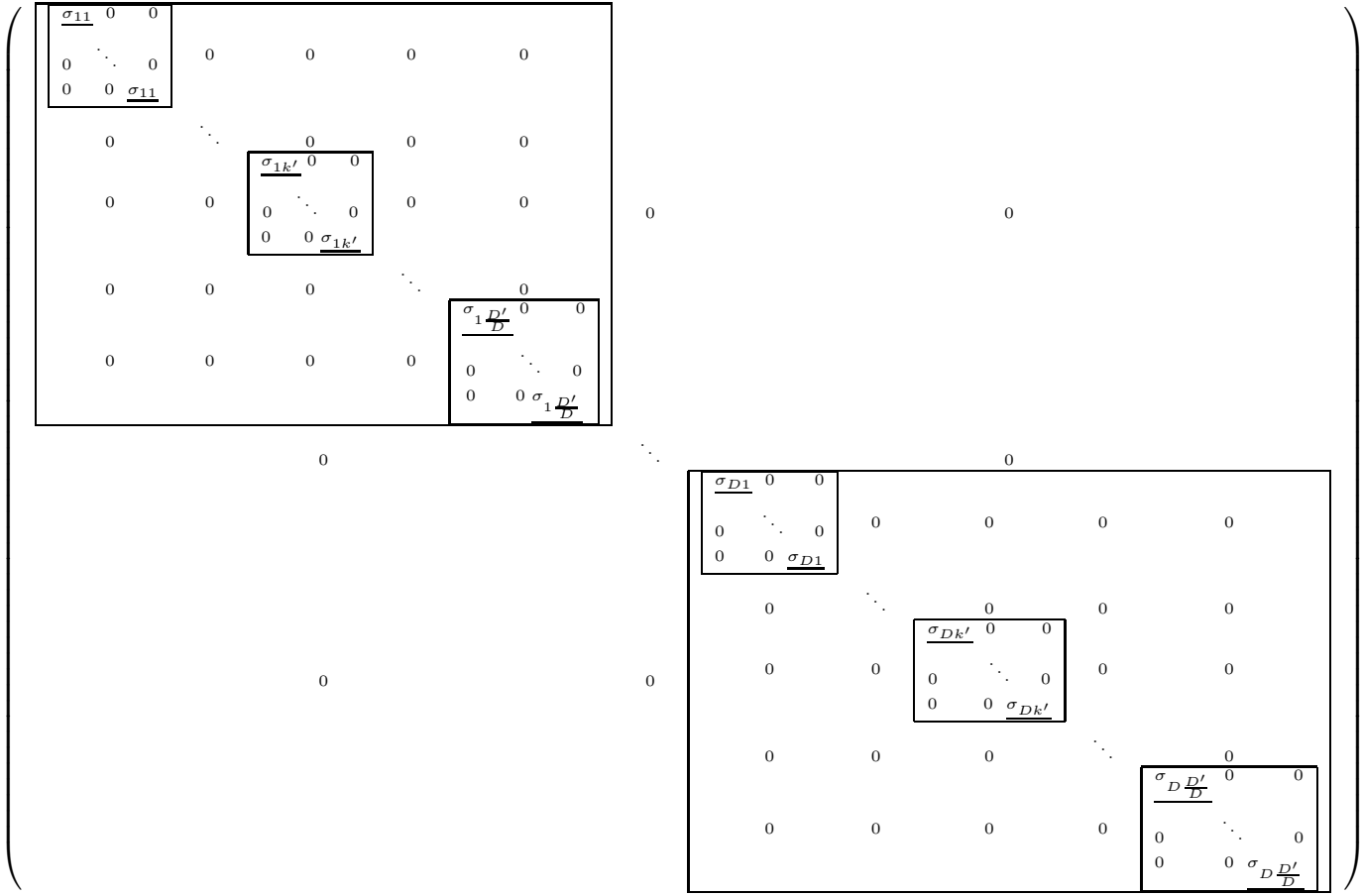
\begin{Tiny}

$\left( \begin{array}{ccc}
\begin{array}{|ccccc|}
\hline
\begin{array}{|c|}
\hline
\underline{\sigma_{11}}  ~~0~~~~~ 0 \\[1ex] 

0~~~\ddots ~~~~ 0 \\[1ex] 

0~~~~~0~~\underline{\sigma_{11}}  \\[1ex] 

\hline
\end{array} & 0 & 0 &0&0\\
0&\ddots&0&0&0\\
0 &0& \begin{array}{|c|}
\hline
\underline{\sigma_{1k'}} ~ 0 ~~~~~0 \\[1ex] 

0~~~\ddots ~~~~0 \\[1ex] 

0~~~~~0~\underline{\sigma_{1k'} } \\[1ex] 
\hline
\end{array} & 0&0 \\
0&0&0&\ddots&0\\
0 & 0 &0&0&  \begin{array}{|c|}
\hline
\underline{\sigma_{1\frac{D'}{D}}} ~ 0 ~~~~~~~0 \\[1ex] 

0~~~~~~\ddots  ~~~~0\\[1ex] 

0~~~~~~0~\underline{\sigma_{1\frac{D'}{D}} }  \\[1ex] 
\hline
\end{array} \\
        \hline
\end{array} & 0 & 0 \\

0 & \ddots & 0 \\
0 & 0 & \begin{array}{|ccccc|}
\hline
\begin{array}{|c|}
\hline
\underline{\sigma_{D1}} ~~ 0 ~~~~~~0 \\[1ex] 

0~~~~~\ddots  ~~~~0 \\[1ex] 

0~~~~~~0~~\underline{\sigma_{D1}}  \\[1ex] 

\hline
\end{array} & 0 & 0 &0&0\\
0&\ddots&0&0&0\\
0 &0& \begin{array}{|c|}
\hline
\underline{\sigma_{Dk'}} ~~ 0 ~~~~~~~0 \\[1ex] 

0~~~~~~~\ddots ~~~~0 \\[1ex] 

0~~~~~~~0~~\underline{\sigma_{Dk'} } \\[1ex] 
\hline
\end{array} & 0&0 \\
0&0&0&\ddots&0\\
0 & 0 &0&0&  \begin{array}{|c|}
\hline
\underline{\sigma_{D\frac{D'}{D}}} ~~ 0 ~~~~~~~~~0 \\[1ex] 

0~~~~~~~~\ddots  ~~~~~~~0\\[1ex] 

0~~~~~~~~~0~~\underline{\sigma_{D\frac{D'}{D}}  } \\[1ex] 
\hline
\end{array} \\
        \hline
\end{array}  \\
\end{array} \right)
$

\end{Tiny}

où l'on a posé $\sigma_{kk'}(sr(c))=\underline{\sigma_{kk'}}$.

Soit $a\in \phi ( G,T,\bar{F}) \setminus \phi( G',T,\bar{F})$, soit $H_a=da^{\vee}(1)$ \cite[§6]{Yu}. Il existe $k\in\{1,\ldots ,D\} $ et $k_1' \not = k_2' \in \{1,\ldots ,D'/D\}$ tel que $H_a = \mathrm{diag}(0,\ldots,0,1,0,\ldots,0,-1,0,\ldots,0)$ et les coefficients $1$ et $-1$ sont aux places $kk'_1$ et $kk'_2$ respectivement.

On a  $X^*(H_a)=\mathrm{tr}( \iota(sr(c) H_a)=\sigma_{kk'_1}(sr(c))-\sigma_{kk'_2}(sr(c))$, il suffit donc de montrer que $\mathrm{ord}(\sigma _{kk'_1}(sr(c))-\sigma_{kk'_2}(sr(c))=\mathrm{ord}(sr(c))$.

Cela résulte du lemme suivant:
\begin{lemm} Soit 
 \begin{small} \shorthandoff{;:!?} \xymatrix @!=0,01cm{E' \ar @{-}[d]^{\not = }\\ E\ar @{-}[d] \\ F } \end{small} une tour d'extensions modérément ramifiées.
Soit $c\in E'$ minimal sur $E$. Soit $\tau: E\to \overline{F}$ un $F-$plongement. Soient $\sigma_1 \not =\sigma_2$, $E'\to \overline{F}$ deux $F-plongements$ tel que $\sigma _1 \mid _E = \sigma _2 \mid _E.$

Alors $\mathrm{ord}((\sigma_1(sr(c))-\sigma_2(sr(c)))=\mathrm{ord}(sr(c))
$
\end{lemm}
\begin{proof}
Soit $\overline{E'}^{gal}$ la clôture galoisienne de $E'$.

Il existe $\tilde{\sigma}_1$ et $\tilde{\sigma}_2$ : $\overline{E'}^{gal}\to \overline{E'}^{gal}$ prolongeant $\sigma_1$ et $\sigma _2$. Il existe aussi $\tilde{\tau} : \overline{E'}^{gal}\to \overline{E'}^{gal}$  prolongeant  $\tau$. On a alors:
\begin{align*}
\mathrm{ord}((\sigma_1(sr(c))-\sigma_2(sr(c)))&=\mathrm{ord}((\tilde{\sigma_1}(sr(c))-\tilde{\sigma_2}(sr(c)))\\
&=\mathrm{ord}(\tilde{\tau}^{-1}(\tilde{\sigma_1}(sr(c))-\tilde{\sigma_2}(sr(c)))\\
&=\mathrm{ord}(\tilde{\tau}^{-1}\tilde{\sigma_1}(sr(c))-\tilde{\tau}^{-1}\tilde{\sigma_2}(sr(c)))
\end{align*}
Puisque $sr(c)$ génère l'extension $E'/E$ d'après \ref{minigene}, on a  $\sigma _1 (sr(c)) \not = \sigma _2 (sr(c))$.
Or $\tilde{\tau}^{-1}\tilde{\sigma_1}$ et $\tilde{\tau}^{-1}\tilde{\sigma_2}$ sont deux éléments de $Gal(\overline{E'}^{gal}/E)$ tel que $\tilde{\tau}^{-1}\tilde{\sigma_1}(sr(c))\not = \tilde{\tau}^{-1}\tilde{\sigma_2}(sr(c))$.
Le résultat découle alors de  \ref{CE}.
\end{proof}
 On a ainsi montré que la condition (GE1) est vérifiée, la condition (GE2) est alors vérifiée d'après \cite[8.1]{Yu}, en effet il n'y a pas de premier de torsion pour $\psi(G)^{\vee}$ lorsque $G= \mathrm{GL}_N$ puisque alors $\psi(G)^{\vee}$ est de type $A$.
\end{proof}
    
   \end{proof}
   
  \section{Donnée de Yu associée et $\beta-extension$}\label{yube}
   
   Soit $[\A,n,0,\beta]$ une strate simple maximale\footnote{Une strate est dite maximale si $\mathfrak{A}$ est un $\mathfrak{o}_{F[\beta ]}-$ordre maximal} tel que $F[\beta]/F$ est modérément ramifiée. Soit $\theta \in \mathcal{C}(\A,0,\beta) $ un caractère simple. Soient $\phi_0 \circ det ,\ldots , \phi_s \circ det$, $y$ et $\overrightarrow{G}$ les caractères génériques, le point de l'immeuble et la suite de Levi tordue modérée associée à $\theta$ (cf \ref{factA}, \ref{factB}, \ref{filtA}, \ref{filtB}), si $s\not = d$ (Cas B), posons $\boldsymbol{\phi}_d  =1$, enfin posons $\boldsymbol{\phi}_i = \phi _i \circ det$ pour $ 0 \leq i \leq s$ . Posons $\overrightarrow{\boldsymbol{\phi}}=(\boldsymbol{\phi}_0   ,\ldots ,\boldsymbol{\phi} _d )$ et $\overrightarrow{\mathbf{r}}=(\mathbf{r}_0 , ... ,\mathbf{r} _d ) $\footnote{ On a mis en gras les caractères $\boldsymbol{\phi}$ "coté Yu".}

   \begin{prop}\label{grou} On a des égalités de groupes:
   \begin{enumerate}
   \item[($i$)]$ H^1(\beta,\A)=K_+^d$
   
   \item[($ii$)] $J^0(\beta,\A)= {}^{\circ}K^d $
   
   \item[($iii$)] $E^{\times} J^0 (\beta, \A) = K^d$
   \end{enumerate}
   \end{prop}

   \begin{proof}
   
   ($i$)On a dans le (Cas A) \begin{center}  $H^1(\beta , \A ) = (1+\p _0 ^1) (1+\p _1 ^{[\frac{r_1}{2}]+1})\cdots (1+ \p ^{[\frac{r_d}{2}]+1})$ \end{center}
    et dans le (Cas B) \begin{center}
    $H^1(\beta,\A)=(1+\p _0^{[\frac{0}{2}+1]})(1+\p _1 ^{[\frac{r_1}{2}] +1}) \cdots (1+\p _i ^{[\frac{r_i}{2}]+1}) \cdots (1+ \p^{[\frac{n}{2}]+1}).$\end{center} 
    
    Or \begin{center}
   
   $K^d _+ = G^0(F)_{y,0+} G^1(F)_{y,\mathbf{s}_0+}\cdots G^d(F)_{y,\mathbf{s}_{d-1}+}, $ \end{center}
   
   ($i$) résulte alors de  \ref{filtA} et \ref{filtB}.
   
   ($ii$)On a  dans le (Cas A) \begin{center}  $J^0(\beta , \A ) = \A _0 ^{\times} (1+\p _1 ^{[\frac{r_1+1}{2}] }) \cdots (1+\p _i ^{[\frac{r_i +1}{2}]}) \cdots (1+ \p^{[\frac{r_d +1}{2}]})$ \end{center}
   et dans le (Cas B) \begin{center}
   
    $J^0(\beta,\A)=\A _0 ^{\times} (1+\p _1 ^{[\frac{r_1+1}{2}] }) \cdots (1+\p _i ^{[\frac{r_i +1}{2}]}) \cdots (1+ \p^{[\frac{n +1}{2}]}).$\end{center}
   
   Or \begin{center}
   $K^d _+ = G^0(F)_{y,0} G^1(F)_{y,\mathbf{s}_0}\cdots G^d(F)_{y,\mathbf{s}_{d-1}} ,$ \end{center}
   
   ($ii$) résulte alors de  \ref{filtA} et \ref{filtB}.
   
   ($iii$) On a  dans le (Cas A)\begin{center} $E^{\times}J^0(\beta,\A)=E^{\times}\A _0 ^{\times} (1+\p _1 ^{[\frac{r_1+1}{2}] }) \cdots (1+\p _i ^{[\frac{r_i +1}{2}]}) \cdots (1+ \p^{[\frac{r_d +1}{2}]})$\end{center}
   
   et dans le (Cas B) \begin{center}$E^{\times}J^0(\beta,\A)=E^{\times}\A _0 ^{\times} (1+\p _1 ^{[\frac{r_1+1}{2}] }) \cdots (1+\p _i ^{[\frac{r_i +1}{2}]}) \cdots (1+ \p^{[\frac{n +1}{2}]}).$\end{center}
   
   Or \begin{center}
   
   $K^d _+ = G^0(F)_{[y]} G^1(F)_{y,\mathbf{s}_0}\cdots G^d(F)_{y,\mathbf{s}_{d-1}} $
   \end{center}
   
  et $E^{\times}\A _0 ^{\times}=G^0(F)_{[y]}$, ($iii$) résulte alors de \ref{filtA} et \ref{filtB}.
   
   \end{proof}
   
   \begin{prop} Soit $\hat{\boldsymbol{\phi} _i }$ le caractère de $K_+^d$ construit par $Yu$ à partir de $\boldsymbol{\phi} _i$ pour $0\leq i \leq d$ (cf §\ref{yu}). On a des égalités de caractères \begin{enumerate}
   \item[($i$)]$\hat{\boldsymbol{\phi} _i}=\theta ^i$ pour $ 0 \leq i \leq s$.
   
   \item[($ii$)]$\prod\limits_{i=0}^d\hat{\boldsymbol{\phi} _i}= \theta$
\end{enumerate}   \end{prop}
   \begin{proof} $\widetilde{HY}$

   ($i$) Remarquons que d'après \ref{grou} les caractères sont bien définis sur le même groupe.
   D'après la définition \cite[§4]{Yu} on a :
   \begin{center}
   $\hat{\phi _i }\mid _{G^0(F)_{y,0+}\cdots  G^i(F)_{y,\mathbf{s}_{i-1}+}}(1+x)=\phi _i \circ det (1+x)$
   
   $\theta ^i \mid _{G^0(F)_{y,0+}\cdots  G^i(F)_{y,\mathbf{s}_{i-1}+}}(1+x)=\phi _i \circ det (1+x)$
   \end{center}

  On a \cite[§4]{Yu} $G^i(F)_{\mathbf{s}_i+ : \mathbf{r} _i +} \simeq \mathfrak{g}^i(F)_{\mathbf{s} _i + : \mathbf{r} _ i + } \subset \mathfrak{g}^i(F)_{\mathbf{s} _i + : \mathbf{r} _ i + } \oplus \mathfrak{n}^i (F)_{\mathbf{s} _ i + : \mathbf{r} _i + }.$ Le premier isomorphisme est induit par l'application $1+x \mapsto x$.
  L'espace $\mathfrak{g}(F)\simeq \mathrm{M}_N(F)$ est muni de la forme bilinéaire $(x,y)\mapsto \mathrm{ tr}(xy)$ et on a  $\mathfrak{n}^i(F)_{\mathbf{s}_i+} \subset \mathfrak{g}^i (F) ^{\perp}$ (\ref{nigl}).
  
  Soit $x\in \mathfrak{g}^i(F)_{\mathbf{s} _i +}, $ écrivons 
   $x=\pi _{\mf{g} _i}(x)+ \pi _{\mf{n} _i}(x)$ la décomposition ci dessus de $x$.

   On a alors \begin{align*}
  \hat{\phi _i}\mid _{G^{i+1}(F)_{y,\mathbf{s}_i}\cdots G^d(F)_{y,\mathbf{s}_{d-1}}}(1+x)&=\hat{\phi _i}\mid _{G^{i+1}(F)_{y,\mathbf{s}_i}\cdots G^d(F)_{y,\mathbf{s}_{d-1}}}(1+\pi_{\mf{g} _i} (x))\\&=\theta ^i \mid _{G^{i+1}(F)_{y,\mathbf{s}_i}\cdots G^d(F)_{y,\mathbf{s}_{d-1}}}(1+\pi_{\mf{g} _i} (x))\\&= \psi \circ \mathrm{tr}( c_i \pi_{\mf{g} _i} (x) ).
  \end{align*}\\
  D'autre part  $\theta ^i \mid _{G^{i+1}(F)_{y,\mathbf{s}_i}\cdots G^d(F)_{y,\mathbf{s}_{d-1}}}(1+x)=\psi \circ \mathrm{tr} (c_i x)$.
  
  Il suffit donc de voir que $\mathrm{tr}( c_i \pi_{\mf{g} _i} (x) )=\mathrm{tr} (c_i x)$, cela résulte du fait que $\pi _{\mf{n} _i}(x)\subset \mathfrak{g}^i (F) ^{\perp} = A_i ^{\perp} $ pour la forme bilinéaire symétrique $(x,y) \mapsto \mathrm{tr}(xy)$.

   ($ii$) immédiat d'après ($i$), \ref{factA}, \ref{factB}, et puisque dans le (Cas B) on a posé $\boldsymbol{\phi} _d = 1$.
   \end{proof}
   
   \begin{prop} La représentation ${}^{\circ}\lambda ={}^{\circ} \lambda(\overrightarrow{G},y,\overrightarrow{\mathbf{r}},\overrightarrow{\boldsymbol{\phi}} )$ de ${}^{\circ} K^d= J^0(\beta,\A)$ associée à $(\overrightarrow{G},y,\overrightarrow{\mathbf{r}},\overrightarrow{\boldsymbol{\phi}} )$ (cf §\ref{buku}) est une $\beta -extension$ de $\theta$.
   \end{prop}
   
   \begin{proof} D'après \ref{beta}, il suffit de montrer que:
   
   \begin{enumerate}
   \item[($a$)] $A_0 ^{\times}=G^0(F)$ entrelace ${}^{\circ} \lambda $ 
   \item[($b$)]  ${}^{\circ} \lambda $  contient $\theta= \prod\limits_{i=1}^d \hat{\boldsymbol{\phi}}$
   \item[($c$)]  $dim ({}^{\circ} \lambda) = dim (\eta) = [J^1(\beta, \A) : H^1(\beta, \A )]^{\frac{1}{2}}$
   \end{enumerate}
   
   ($a$) résulte de \cite[15.6]{Yu}
   
   ($b$) résulte de \cite[4.4]{Yu}
   
   ($c$)\footnote{Attention à ne pas confondre le groupe $J^1(\beta , \A ) $ défini dans \cite{BK} avec les groupes $J^i$ et $J^i_+$ pour $i=1$ défini dans \cite{Yu}} 
   
   ${}^{\circ} \lambda = {}^{\circ} \kappa_0 \otimes {}^{\circ} \kappa _1 \otimes \cdots \otimes {}^{\circ} \kappa _d$
   
   La représentation $\kappa _i$ provient d'une représentation de Heisenberg pour $0\leq i \leq d-1$, sa dimension est $[J^{i+1} : J^{i+1} _+]^{\frac{1}{2}}$  ( cf § \ref{yu}).
   
     $dim (^{\circ} \kappa _i)= dim (\kappa _i) = [J^{i+1} : J^{i+1} _+]^{\frac{1}{2}} $ pour $0\leq i \leq d-1$
   
   $dim (^{\circ} \kappa _i)= 1 $ pour $i=d$
   
   Donc $dim ({}^{\circ} \lambda) = \prod\limits_{i=0}^{d-1} [J^{i+1} : J^{i+1} _+]^{\frac{1}{2}}=\prod\limits_{i=1}^{d} [J^{i} : J^{i} _+]^{\frac{1}{2}}$
   
   Il suffit donc de montrer que $\prod\limits_{i=1}^d [J^i : J^i _+]=[J^1(\beta, \A) : H^1(\beta, \A )]$

  On a (\cite[§4]{Yu})  $J^1(\beta , \A)= G^0_{y,0+}J^1 \cdots J^{d-1}J^{d}$ 
   et $ H^1(\beta, \A )= H^1(\beta _0 , \A) =  G^0_{y,0+}J^1_+\cdots J^{d-1}_+J^d_+$.

   Il suffit donc de montrer que $\prod\limits_{i=1}^d [J^i : J^i _+]=[G^0_{y,0+}J^1\cdots J^{d-1}J^d :  G^0_{y,0+}J^1_+\cdots J^{d-1}_+J^d_+ )]$
   
   Montrons cela par récurrence sur $d$.
   
   Si $d=1$ alors $J^1(\beta , \A)= G^0_{y,0+} G^1_{y,s_0} = G^0_{y,0+}J^1$ et $H^1(\beta , \A) =G^0 _{y,0+} G^1_{y,s_0+}=G^0_{y,0+}J^1_+$ et l'application $J^1 \to G^0_{y,0+}J^1 / G^0_{y,0+}J^1_+$ est surjective de noyau $J^1 \cap G^0_{y,0+}J^1_+ =J^1_+$. Donc la propriété est vrai au rang $d=1$.
   
   Supposons la propriété vrai au rang $d-1$.
   
   On a  donc $\prod\limits_{i=1}^{d-1} [J^i : J^i _+]=[G^0_{y,0+}J^1\cdots J^{d-1} :  G^0_{y,0+}J^1_+\cdots J^{d-1}_+ )]$
  
  Le morphisme de groupe 
  \begin{center}$J^d \to \left( (G^0_{y,0+}J^1\cdots J^{d-1}J^d) /  (G^0_{y,0+}J^1_+\cdots J^{d-1}_+J^d_+ ) \right) /\left( (G^0_{y,0+}J^1\cdots J^{d-1}) /  (G^0_{y,0+}J^1_+\cdots J^{d-1}_+ ) \right)$\end{center} est surjectif de noyau  $J^d_+$.
  
  d'où ($c$).

\end{proof}

Puisque la strate $[\A,n,0,\beta] $ est supposée maximale, on a $J^0(\beta, \A) / J^1(\beta , \A) \simeq \A _0 ^{\times} / \left(1+\mathfrak{P}_0 \right)\simeq  \mathrm{GL}(f,k_E)$. Soit $\sigma$ une représentation irréductible cuspidale de $\mathrm{GL}(f,k_E)$.

Il existe une unique représentation $\rho$ de $E^{\times} \A _0 ^{\times} = G^0(F)_{[y]}$ qui prolonge $\sigma$, de plus $\mathrm{c-ind}_{G^0(F)_{[y]}}^{G^0(F)} (\rho)$ est irréductible et supercuspidale.

Le quintuplet ($\overrightarrow{G},y,\rho, \overrightarrow{\mathbf{r}}, \overrightarrow{\boldsymbol{\phi}}$) est une donnée de Yu générique, en effet :
\begin{enumerate}
\item[($\overrightarrow{G}$)]  est une suite de Levi tordue modérée tel que $Z(G^0)/Z(G)$ est anisotrope 
\item[($y$)] est un sommet de $\mathcal{B}(G^0,F)$ puisque $\A _0$ est un $\mathfrak{o}_{E_0}$-ordre maximal

\item[($ \overrightarrow{\mathbf{r}}$)] satisfait bien la condition désirée, dans le (Cas A) on a $0<\mathbf{r}_0 <\mathbf{r} _1 <\ldots <\mathbf{r}_{d-1} <\mathbf{r} _{d}$ et dans le (cas B) on a $0<\mathbf{r}_0 <\mathbf{r} _1 <\ldots <\mathbf{r}_{d-1} =\mathbf{r} _{d}$.

\item[($\rho$)] vérifie la condition requise d'après ce qui précède ci-dessus

\item[($\overrightarrow{\boldsymbol{\phi}})$]  satisfait les propriétés voulues en vertu de \ref{cage} et du fait que dans le (cas B) on a $\mathbf{r}_{d-1} =\mathbf{r} _{d}$ et $\boldsymbol{\phi} _d =1$.

\end{enumerate}

Ainsi, ${}^{\circ} \rho _d (\overrightarrow{G},y,\rho,\overrightarrow{\mathbf{r}} , \overrightarrow{\boldsymbol{\phi}})$ est un type simple  et $\rho _d (\overrightarrow{G},y,\rho,\overrightarrow{\mathbf{r}}, \overrightarrow{\boldsymbol{\phi}})$ est égale à la représentation $\Lambda$ de $E^{\times } J^0(\beta , \A)$ correspondant au type simple $\sigma \otimes {}^{\circ} \lambda$.

\begin{rema} La comparaison que l'on a faite entre les constructions des représentations supercuspidales de $\mathrm{GL}_N$ de Bushnell-Kutzko (\cite{BK}) et Yu (\cite{Yu}) devrait pouvoir  se prolonger aux constructions de Blondel-Blasco (\cite{blbl}), Bushnell-Kutzko (\cite{bukusl}), Ngô (\cite{vdng}), Sécherre et Stevens (\cite{Sech}, \cite{Stev}).

De plus on devrait pouvoir comparer de la même manière les constructions de types généraux au sens de  Bushnell-Kutzko \cite{bkty}. On pourrait ainsi comparer  \cite{bkge} et \cite{goro} avec \cite{ykge}.
\end{rema}

\bibliographystyle{plain-fr}

\bibliography{library}

 \end{document}